\documentclass{default}

\title{Gaussian distribution of short sums of trace functions over finite fields}
\def\titleHead{Gaussian distribution of short sums of trace functions}
\author{Corentin Perret-Gentil}
\address{ETH Zürich, Department of Mathematics}
\email{corentin.perretgentil@\{math.ethz.ch,gmail.com\}}
\hypersetup{pdfauthor={Corentin Perret-Gentil}}

\date{\today}

\begin{document}

\begin{abstract}
  We show that under certain general conditions, short sums of $\ell$-adic trace functions over finite fields follow a normal distribution asymptotically when the origin varies, generalizing results of Erd\H{o}s-Davenport, Mak-Zaharescu and Lamzouri. In particular, this applies to exponential sums arising from Fourier transforms such as Kloosterman sums or Birch sums, as we can deduce from the works of Katz. By approximating the moments of traces of random matrices in monodromy groups, a quantitative version can be given as in Lamzouri's article, exhibiting a different phenomenon than the averaging from the central limit theorem.
\end{abstract}
\maketitle

\tableofcontents

\section{Introduction}

Let $\F_q$ denote the finite field of cardinality $q$ in characteristic $p$. For a function $t: \F_q\to\C$ and a subset $I\subset\F_q$, we let
\[S(t,I)=\sum_{x\in I} t(x)\]
be the partial sum over $I$. For $I$ of various structures and sizes, such sums are omnipresent in analytic number theory (see e.g. \cite[Chapter 12]{IK04}). Due to oscillations, they often exhibit cancellation, and as a general phenomenon we can expect (or wish for) \textit{square-root cancellation} $|S(t,I)|\ll \sqrt{|I|}$.
\subsection{Sums over subsets with varying origin}
For $x\in\F_q$, we denote by $I+x=\{y+x : y\in I\}$ the translate of $I$ by $x$. Given a family of functions $(t_q: \F_q\to\C)_q$ and intervals $I_q\subset\F_q$, we are interested in the distribution of the complex random variable
\begin{equation}
  \label{eq:RVSfI}
  \left(\frac{S(t_q,I_q+x)}{\sqrt{|I_q|}}\right)_{x\in\F_q,}
\end{equation}
with respect to the uniform measure on $\F_q$, as $q,|I_q|\to\infty$.

\begin{example}\label{ex:1}
  When $q=p$, the finite field $\F_p$ can be identified with the discrete interval $[1\dots p]$. For an interval $I_H=[1\dots H]\subset[1\dots p]$ and $1\le x\le p$ an integer, $S(t,I_H+x)$ is the partial sum
  \[S(t,x,H):=\sum_{x\le y< x+H} t(y)\]
  of length $H$ starting at $x$. More generally, when $q=p^e$, we can consider ``boxes'' in $\F_q\cong\F_p^e$.
\end{example}

\subsection{The case of Dirichlet characters}\label{subsec:DirichlarCharacters}
In the situation of Example \ref{ex:1} with $(t_p)_p=(\chi_p)_p$ a family of Dirichlet characters, the question of the distribution of the random variable \eqref{eq:RVSfI} appears in the literature as follows:
\begin{itemize}[leftmargin=*]
\setlength{\itemsep}{0.2cm}
\item When $\chi_p$ is the Legendre symbol, Davenport and Erdős \cite{DavenportErdos} showed that the real-valued random variable
 \[\left(S(\chi_p,x,H_p)/\sqrt{H_p}\right)_{x\in\F_p}\]
converges in law to a normal distribution with mean $0$ and unit variance when
\begin{equation}
\label{eq:Hp}
p, \ H_p\to\infty\text{ with }\log{H_p}=o(\log{p}).
\end{equation}
\item Mak and Zaharescu \cite{MakZahShortSums} generalized this result to short sums of the form
\[\tilde S_p(x,H_p)=\sum_{\substack{P=(x_1,x_2)\in C\\ x\le x_1<x+H_p\\ x_2\in I}}\chi_p(g(P))\psi_p(t(P)),\]
where $C$ is an absolutely irreducible affine plane curve over $\F_p$, $g,f\in\F_p(x,y)$ are rational functions, $\psi_p$ (resp. $\chi_p$) is an additive (resp. non-real multiplicative) character modulo $p$, and $I$ is an interval. Under some technical conditions, they similarly obtain that the projection of the random variable $(\tilde S_p(x,H_p)/\sqrt{H_p})_{x\in\F_p}$ on any line through the origin converges in law to a normal distribution with mean $0$ and unit variance when $p,H_p\to\infty$ under \eqref{eq:Hp}.
\item Lamzouri \cite{LamzShortSums} showed that when $\chi_p$ is a non-real Dirichlet character, the random variable \[\left(S(\chi_p,x,H_p)/\sqrt{H_p}\right)_{x\in\F_p}\] converge in law to a normal distribution in 
$\C$ with mean $0$ and covariance matrix $\frac{1}{2}\left(\begin{smallmatrix}
  1&0\\0&1
\end{smallmatrix}\right)$ when $p,H_p\to\infty$ with \eqref{eq:Hp}.
\end{itemize}
\vspace{0.1cm}

All of the above proceed by using the method of moments. To do so, one needs bounds on character sums that follow from the work of Weil on the Riemann hypothesis for curves over finite fields.

A particular aspect of Lamzouri's method in \cite{LamzShortSums} is to consider a \textit{probabilistic model}, where the values of a multiplicative character are modeled as independent random variables uniformly distributed on the unit circle in $\C$. This model is shown to be accurate (in the sense of convergence in law) by bounding an exponential sum.

\subsection{Generalization to trace functions}
In this article, we will consider the question introduced above for families $(t_q: \F_q\to\C)_q$ of \textit{$\ell$-adic trace functions over $\F_q$}, as they appear in particular in the works of Katz (see for example \cite{KatzGKM} and \cite{KatzESDE}), and more recently in the series of papers by Fouvry, Kowalski, Michel and others (see \cite{FKMPisa}, \cite[Section 6]{Polymath8a} or \cite{PG16} for surveys).\\

Using the results surveyed in \cite{FKMSumProducts}, building upon Deligne's generalization of the Riemann Hypothesis over finite fields \cite{Del2} and the works of Katz, we will show that under general assumptions on a family of $\ell$-adic trace functions $(t_q: \F_q\to\C)_{q}$, and a family of sets $I_q\subset\F_q$, the random variable \eqref{eq:RVSfI} converges in law to a normal distribution in $\C\cong\R^2$ when $q, \ H_q=|I_q|\to\infty$ in the range \eqref{eq:Hp}. Hence, we generalize the results of Section \ref{subsec:DirichlarCharacters} to trace functions.\\

For example, for the (normalized) Kloosterman sums of rank $n\ge 2$
\begin{eqnarray}
    t_q(x)&=&\Kl_{n,q}(x)\nonumber\\
    &=&\frac{(-1)^{n-1}}{q^{(n-1)/2}}\sum_{\substack{x_1,\dots,x_n\in\F_q^\times\\ x_1\cdots x_n=x}}e\left(\frac{\tr(x_1+\dots+x_n)}{p}\right)\label{eq:KS}
  \end{eqnarray}
  (where $\tr:\F_q\to\F_p$ is the trace), we get the following:
\begin{theorem*}
  Let $n\ge 2$ and for every prime power $q$, let $I_q\subset\F_q$. The complex random variable
  \[\left(\frac{S(\Kl_{n,q},I_q+x)}{\sqrt{|I_q|}}\right)_{x\in\F_q}\]
   (with respect to the uniform measure on $\F_q$) converges in law to a normal distribution $\Nc$ in $\C\cong\R^2$, with mean 0 and covariance matrix $\left(\begin{smallmatrix}
    1&0\\0&0
  \end{smallmatrix}\right)$ if $n$ is even and $\frac{1}{2}\left(\begin{smallmatrix}
    1&0\\0&1
  \end{smallmatrix}\right)$ if $n$ is odd, when $q,|I_q|\to\infty$ with $\log{|I_q|}=o(\log{q})$.

More precisely, for any $\varepsilon\in(0,1/2)$ and for any closed rectangle $A\subset\C$ with sides parallel to the coordinate axes and Lebesgue measure $\mu(A)$, the probability
  \[P\left(\frac{S(\Kl_{n,q},I_q+x)}{\sqrt{|I_q|}}\in A\right)=\frac{|\{x\in\F_q : S(\Kl_{n,q},I_q+x)/|I_q|^{1/2}\in A\}|}{q}\]
  is given by
\[P(\Nc\in A)+O_\varepsilon\left(\mu(A)\left(q^{-\frac{1}{2}+\varepsilon}+\left(\frac{\log{|I_q|}}{\log{q}}\right)^{2/5}+\frac{1}{\sqrt{|I_q|}}\right)\right)\]
when $q, \ |I_q|\to\infty$ with under the range $\log{|I_q|}=o(\log{q})$ if $n$ is even and $|I_q|=o \left((\log{q})^{\frac{3}{2(1+\varepsilon)}}\right)$ otherwise. As $n\to\infty$ or if $n$ is even, the exponents $2/5$ and $3/2$ can be replaced by $1/2$ and $1$, respectively.
\end{theorem*}

The general results will be stated in Section \ref{sec:statements}.

\subsubsection{Examples}\label{subsubsec:examplesTF}

Examples of $\ell$-adic trace functions over $\F_q$ we will consider include:
\begin{enumerate}[(a)]
\item\label{item:exDirichlet} Dirichlet characters $\chi$ modulo $q$ or compositions $\chi\circ f$, where $f\in\F_q(T)$ is a rational function. This is the case considered in
 \cite{LamzShortSums} (if $f=\id$) and \cite{MakZahShortSums}, when $q=p$.
\item\label{item:exKS} Hyper-Kloosterman sums $\Kl_{n,q}$ of rank $n\ge 2$, or more generally hypergeometric sums, as studied by Katz in \cite{KatzGKM} and \cite{KatzESDE}.
\item\label{item:exGenES} General exponential sums of the form
  \begin{equation}
    \label{eq:expsumGen}
    t_q(x)=\frac{1}{\sqrt{q}}\sum_{y\in\F_q} e\left(\frac{\tr(xf(y)+h(y))}{p}\right)\chi(g(y)),
  \end{equation}
  for $f,g,h\in\Q(X)$ rational functions and $\chi: \F_q^\times\to\C$ a multiplicative character. This includes Birch sums
  \begin{equation}
    \label{eq:BS}
    t_q(x)=\operatorname{Bi}(x,q)=\frac{1}{\sqrt{q}}\sum_{y\in\F_q} e \left(\frac{\tr(xy+y^3)}{p}\right),
  \end{equation}
  considered by Birch, Livné and Katz, and sums of the form 
  \begin{equation}
    \label{eq:FouvryMichel}
    t_q(x)=\frac{1}{\sqrt{q}}\sum_{y\in\F_q}e\left(\frac{\tr(xf(y))}{p}\right),
  \end{equation}
  studied by Katz and Fouvry-Michel (see e.g. \cite{MichelMinorationsSommesExp}).
\item\label{item:exCP} Functions counting points on families of curves over $\F_q$ parametrized by varieties over $\F_q$, as surveyed in \cite[Chapter 10]{KatzSarnak91}.
\end{enumerate}

Note that $t_q$ can be complex or real-valued (the latter occurring for example for hyper-Kloosterman sums of even rank and Birch sums).

\begin{acknowledgements}
The author would like to thank his supervisor Emmanuel Kowalski for guidance and advice during this project. It is a pleasure to acknowledge in particular the influence of the works of \'Etienne Fouvry, Nicholas Katz, Emmanuel Kowalski, Youness Lamzouri and Philippe Michel. The computations present in this document have been performed with the SageMath \cite{Sage} software. This work was partially supported by DFG-SNF lead agency program grant 200021L\_153647. The results also appear in the author's PhD thesis \cite{PG16}.
\end{acknowledgements}

\section{Statement of the results}\label{sec:statements}

\subsection{Trace functions over finite fields}

We briefly recall some definitions and terminology around $\ell$-adic trace function over finite fields necessary to state our results, and refer the reader to \cite{KatzGKM}, \cite[Chapter 7]{KatzESDE}, \cite[Section 6]{Polymath8a}, \cite{FKMPisa}, \cite{FKMCours} or \cite{PG16} for details and further references.

\begin{definition}
  Let $\ell$ be a prime number distinct from the characteristic $p$ of the finite field $\F_q$. We call \textit{$\ell$-adic sheaf over $\F_q$} a constructible sheaf $\Fc$ of $\overline\Q_\ell$-modules on $\P^1/\F_q$ (with respect to the étale topology) which is middle-extension, i.e. for every nonempty open $j: U\to \P^1$ on which $\Fc$ is lisse, we have $\Fc\cong j_*j^*\Fc$. We write $\Sing(\Fc)=\P^1(\overline\F_q)-U_\Fc(\overline\F_q)$ for the set of \textit{singularities} of $\Fc$, where $U_\Fc$ is the maximal open set of lissity\footnote{One shows that such an open exists -- it is where the stalk has generic rank -- and that $\Fc$ is determined by its restriction to $U_\Fc$, see e.g. \cite[8.5.1]{KatzGKM}.} of $\Fc$.
\end{definition}

There is an alternative point of view through $\ell$-adic representations of étale fundamental groups that can be very convenient in practice:

\begin{proposition}
  There is an equivalence of categories between $\ell$-adic sheaves $\Fc$ over $\F_q$ and continuous finite-dimensional $\ell$-adic representations
  \[\rho_\Fc: \pi_{1,q}:=\Gal\big(\F_q(T)^\sep/\F_q(T)\big)\to\GL(\Fc_{\overline\eta})\cong\GL_n(\overline\Q_\ell).\]
  Moreover, $\Fc$ is lisse at $x\in \P^1(\overline\F_q)$ if and only if the inertia group $I_x\le\pi_{1,q}$ acts trivially on $\overline\Q_\ell^n$. The integer $n$ is the \emph{rank} of $\Fc$.
\end{proposition}

\begin{definition}
  Let $\iota:\overline\Q_\ell\to\C$ be a fixed isomorphism of fields. The \textit{trace function} associated to an $\ell$-adic sheaf $\Fc$ over $\F_q$ corresponding to the representation $\rho_\Fc: \pi_{1,q}\to\GL(V)$ is the function
  \begin{eqnarray*}
    t_\Fc: \F_q&\to&\C\\
    x&\mapsto&\iota\tr \left(\rho_\Fc(\Frob_{x,q}) \mid V^{I_x}\right),
  \end{eqnarray*}
  where $\Frob_{x,q}\in(D_x/I_x)^\sharp\cong\Gal(\overline\F_q/\F_q)^\sharp$ is the geometric Frobenius at $x\in\F_q$, for $D_x\le\pi_{1,q}$ the decomposition group at $x$.
\end{definition}

\begin{definition}
  An $\ell$-adic sheaf $\Fc$ over $\F_q$ is \textit{pointwise pure of weight $0$} if for every finite extension $\F_{q'}/\F_q$ and every $x\in U_\Fc(\F_{q'})$, the eigenvalues of $\rho_{\Fc}(\Frob_{x,q'})$ are Weil numbers of weight $0$, i.e. their images through any isomorphism of fields $\iota: \overline\Q_\ell\to\C$ have unit absolute value.
\end{definition}
By a result of Deligne \cite[1.8]{Del2}, we have $||t_\Fc||_\infty\le\rank(\Fc)$ if $\Fc$ is pointwise pure of weight $0$ (this is clear at points of lissity), so the former definition corresponds to a normalization assumption for the trace function.\\

By the Grothendieck-Lefschetz trace formula, the Euler-Poincaré formula of Grothendieck-Ogg-Safarevich and Deligne's generalization of the Riemann hypothesis over finite fields to weights of $\ell$-adic sheaves \cite{Del2}, we have a precise control on sums of trace functions:

\begin{theorem}\label{thm:DeligneSumTF}
  For $\Fc$ an $\ell$-adic sheaf over $\F_q$ that is pointwise pure of weight $0$, we have
  \[\sum_{x\in \F_q} t_\Fc(x)=q\cdot\tr\left(\Frob_q\,|\,\Fc_{\pi_{1,q}^\geom}\right)+O \left(E(\Fc)\sqrt{q}\right),\]
  where 
  \begin{equation*}
    \xymatrix{
      1\ar[r]&\pi_{1,q}^\geom=\Gal \left(\F_q(T)^\sep/\overline\F_q(T)\right)\ar[r]&\pi_{1,q}\ar[r]&\Gal(\overline\F_q/\F_q)\ar[r]&1
    }
  \end{equation*}
  is exact, $\Frob_q\in\Gal(\overline\F_q/\F_q)$ is the geometric Frobenius, $\Fc_{\pi_{1,q}^\geom}$ is the space of coinvariants of the representation $\rho_\Fc$ of $\pi_{1,q}^\geom$, and
  \begin{equation}
    \label{eq:EFc}
    E(\Fc)= \rank(\Fc)\left[|\Sing(\Fc)|-1+\sum_{x\in\Sing(\Fc)}\Swan_{x}(\Fc)\right].
  \end{equation}
\end{theorem}
\begin{proof}
  See \cite[Exposé 6]{DelEC}, \cite[Chapter 4]{FKMCours}, \cite[Chapter 2]{KatzGKM} or \cite[Section 9]{AlgebraicTwists}.
\end{proof}

\begin{remark}\label{rem:conductorGrowing}
  In the works of Fouvry-Kowalski-Michel and others, the error term is usually only given in terms of the \textit{conductor}
\[\cond(\Fc)=\rank(\Fc)+|\Sing(\Fc)|+\sum_{x\in\Sing(\Fc)} \Swan_x(\Fc),\]
which is independent from $q$ in most ``natural'' families of sheaves. We are more precise in \eqref{eq:EFc} to be able to discuss cases where the conductor will be growing.
\end{remark}

\begin{definition}
  An $\ell$-adic sheaf over $\F_q$ is \textit{irreducible} (resp. \textit{geometrically irreducible}) if the corresponding representation of $\pi_{1,q}$ (resp. of $\pi_{1,q}^\geom$) is irreducible.
\end{definition}

Finally, we recall the definition of monodromy groups.
\begin{definition}
  For a fixed isomorphism of fields $\iota: \overline\Q_\ell\to\C$, the \textit{geometric} (resp. \textit{arithmetic}) \textit{monodromy group} of an $\ell$-adic sheaf $\Fc$ over $\F_q$ with rank $n$ is the algebraic group
\[G_\geom(\Fc)=\iota\overline{\rho_\Fc\big(\pi_{1,q}^\geom\big)}\le G_\arith(\Fc)=\iota\overline{\rho_\Fc\big(\pi_{1,q}\big)}\hspace{0.5cm}\le\GL_n(\C),\]
where $\overline{\,\cdot\,}$ denotes Zariski closure.
\end{definition}

\begin{remark}\label{rem:DeligneSumTFMon}
  The main term in Theorem \ref{thm:DeligneSumTF} can be rewritten as $q\tr(\Frob_q\mid \Fc_{G_\geom(\Fc)})$, which is $q\dim(\Fc_G)$ if $G_\geom(\Fc)=G_\arith(\Fc)$.
\end{remark}

\subsection{Coherent families}

Finally, we introduce the class of families of trace functions to which our results will apply.

\begin{definition}\label{def:coherentFamily}
  Let us fix a prime $\ell$ and an isomorphism of fields $\iota:\overline\Q_\ell\to\C$. A family $(\Fc_q)_q$ of pointwise pure of weight $0$ and geometrically irreducible $\ell$-adic sheaves over $\F_q$ (for $q$ varying over powers of primes distinct from $\ell$) is said to be \textit{coherent} if:
  \begin{enumerate}
  \item (Conductor) $\cond(\Fc_q)$ is bounded independently from $q$,
  \end{enumerate}
  and either:
  \begin{enumerate}[start=2]
  \item \textit{Kummer case}: For every $q$, $\Fc_q$ is a Kummer sheaf $\Lc_{\chi_q\circ f_q}$ for a character $\chi_q:\F_q^\times\to\C^\times$ and $f_q\in\F_q(X)$, and the characters $\chi_q$ are either all real-valued or all complex-valued.
  \item \textit{Classical case}: There exists $G\in\{\SL_{n+1}(\C),\Sp_{2n}(\C),\SO_{n+1}(\C)\}-\{\SO_8(\C)\}$ for some $n\ge 1$ such that for every sheaf $\Fc_q$ over $\F_q$ in the family:
    \begin{enumerate}
    \item (Monodromy groups) The geometric and arithmetic monodromy groups of $\Fc_q$ coincide and are conjugate to $G$ in $\GL_n(\C)$.
    \item (Independence of shifts) There is no geometric isomorphism
      \begin{equation}
        \label{eq:translatinIsom}
        [+a]^*\Fc_q\cong\Fc_q\otimes\Lc\hspace{0.25cm}\text{ or }\hspace{0.25cm}[+a]^*\Fc_q\cong D(\Fc_q)\otimes\Lc
      \end{equation}
      for a sheaf $\Lc$ of rank $1$ over $\F_q$ and $a\in\G_m(\F_q)$, where $D(\Fc_q)$ denotes the dual sheaf (corresponding to the dual representation).
    \end{enumerate}
  \end{enumerate}
\end{definition}
\begin{definition}
  For $\Fc$ an $\ell$-adic sheaf over $\F_q$ and $I\subset\F_q$, we say that $\Fc$ is \textit{$I$-compatible} if, in the case where $\Fc$ is a Kummer sheaf $\Lc_{\chi(f)}$ with $\deg(f)>1$, we have that $\sum_{i=1}^m x_i\neq 0$ for all $1\le m\le\deg(f)$ and $x_1,\dots,x_m\in I$. If $\Fc$ is not a Kummer sheaf, it is always $I$-compatible.
\end{definition}
\begin{example}\label{ex:noZeromSum}
  A Kummer sheaf $\Lc_{\chi(f)}$ is $I$-compatible if we have\linebreak$I\subset[1\dots p/\deg(f))^e\subset\F_q\cong\F_p^e$.
\end{example}
\begin{remarks}
  As we shall see, these conditions are fairly generic for natural families arising in number theory. For example, geometric irreducibility and uniform boundedness of conductors are stable by $\ell$-adic Fourier transform. In the classical case, the equality of monodromy groups is to control a main term through monodromy (see Remark \ref{rem:DeligneSumTFMon}), while the other conditions are to show that the monodromy group of a sheaf obtained as a sum of translates of the $\Fc_q$ is as large as possible, through the Goursat-Kolchin-Ribet criterion of Katz.
\end{remarks}

\subsection{Qualitative version}

\begin{theorem}\label{thm:gaussian1}
  Let $(t_q: \F_q\to\C)_q$ be a coherent family of trace functions and let $(I_q)_q$ be a family of subsets $I_q\subset\F_q$ such that $\Fc_q$ is $I_q$-compatible. Then the complex random variable
  \[\left(\frac{S(t_q,I_q+x)}{\sqrt{|I_q|}}\right)_{x\in\F_q}\]
  (with respect to the uniform measure on $\F_q$) converges in law to a normal distribution $\Nc$ in $\C\cong\R^2$, with mean 0 and covariance matrix
  \begin{equation}
    \label{eq:covmatrix}
    \left(\begin{matrix}
    1&0\\0&0
  \end{matrix}\right)\text{ if }t_q\text{ has real values,}\hspace{0.4cm}\frac{1}{2}\left(\begin{matrix}
    1&0\\0&1
  \end{matrix}\right)\text{ otherwise}
  \end{equation}
when $q,|I_q|\to\infty$ with $\log{|I_q|}=o(\log{q})$.
\end{theorem}
\begin{remarks}
  \begin{enumerate}
  \item We do not require that $I_q$ be an interval, but it can rather be \textit{any} (small) subset.
  \item The result shows in particular that the limit has independent real and imaginary parts.
  \item As we shall see, the condition on $t_q$ being real-valued can be reformulated as a condition on the monodromy group of the family.
  \end{enumerate}
\end{remarks}

  To prove this theorem, we extend and adapt the method of \cite{LamzShortSums}. The values of the trace functions are modeled by random variables distributed like traces of random matrices uniform in maximal compact subgroups of the monodromy group with respect to the Haar measure (as in Deligne's equidistribution theorem), and the short sums by random walks.

The $\ell$-adic formalism and Deligne's analogue of the Riemann hypothesis over finite fields applied to sum of products are used to show that this model is accurate, through the method of moments.

The conclusion then follows from the central limit theorem.

We mention that similar ideas are also used in \cite{KowSawin} to study the paths obtained by joining partial Kloosterman and Birch sums, as stochastic processes.

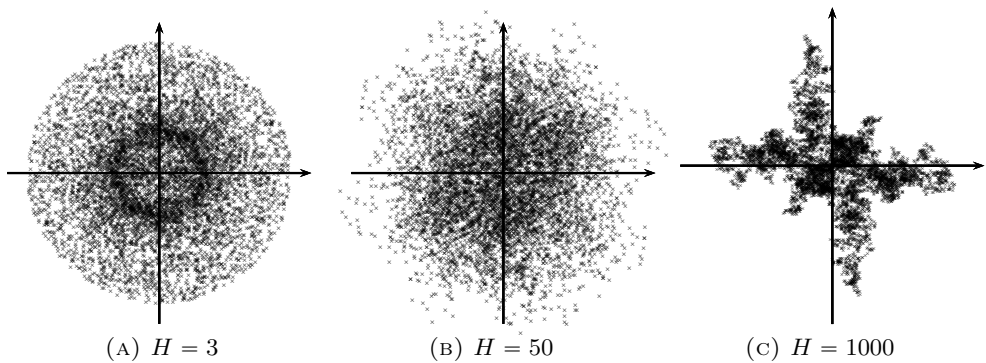
\begin{figure}
  \centering
 \psset{xunit=1cm,yunit=1cm}
  \subfloat[$H=3$]{
    \begin{pspicture}(-2,-2)(2.2,2)
      \psaxes[Dy=50,Dx=0.5,ticks=none,labels=none]{->}(0,0)(-2,-2)(2,2)
      \readdata[nStep=10]{\graph}{data/Dirichlet-p7927-H3.txt}
      \dataplot[plotstyle=dots,dotsize=0.05,dotstyle=x]{\graph}
    \end{pspicture}}
  \psset{xunit=1cm,yunit=1cm}
  \subfloat[$H=50$]{
    \begin{pspicture}(-2.2,-2)(2,2)
      \psaxes[Dy=50,Dx=0.5,ticks=none,labels=none]{->}(0,0)(-2,-2)(2,2)
      \readdata[nStep=10]{\graph}{data/Dirichlet-p7927-H50.txt}
      \dataplot[plotstyle=dots,dotsize=0.05,dotstyle=x]{\graph}
    \end{pspicture}}
  \psset{xunit=1cm,yunit=0.7cm}
  \subfloat[$H=1000$]{
    \begin{pspicture}(-2.2,-3)(2,3)
      \psaxes[Dy=50,Dx=0.5,ticks=none,labels=none]{->}(0,0)(-2,-3)(2,3)
      \readdata[nStep=10]{\graph}{data/Dirichlet-p7927-H1000.txt}
      \dataplot[plotstyle=dots,dotsize=0.04,dotstyle=x]{\graph}
    \end{pspicture}}
  \caption{Distribution of sums of trace functions for a Dirichlet character modulo $p=7927$ of order $p-1$.}
\end{figure}

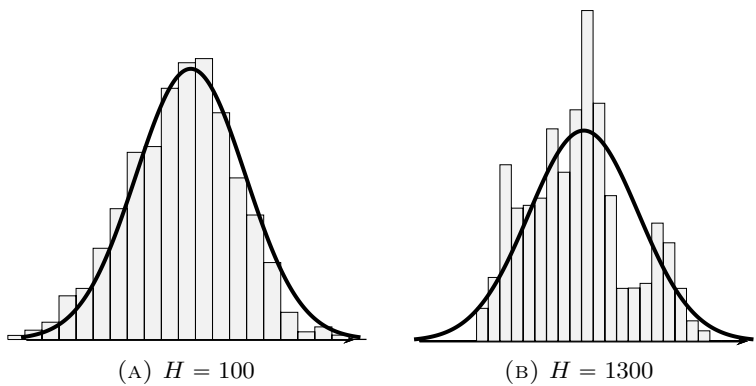
\begin{figure}
  \centering
    \subfloat[$H=100$]{
    \psset{xunit=0.72cm,yunit=9cm}
    \begin{pspicture}(-3.5,-0.01)(3.5,0.45)
      \psaxes[labels=none,ticks=none,Dy=0.25,Dx=2]{->}(0,0)(-3,0)(3,0.3)
      \readdata[nStep=20]{\graph}{data/gaussian-Kl-H100.txt}
      \dataplot[plotstyle=bar,barwidth=0.225,linewidth=0.2pt,fillstyle=solid,fillcolor=vlightgray]{\graph}
      \psplot[linewidth=1.5pt,plotpoints=200]{-3.1}{3.1}{0.3989*2.718^(-x^2/2.0)}
    \end{pspicture}}
      \subfloat[$H=1300$]{
    \psset{xunit=0.72cm,yunit=7cm}
    \begin{pspicture}(-3.5,-0.01)(3.5,0.45)
      \psaxes[labels=none,ticks=none,Dy=0.25,Dx=2]{->}(0,0)(-3,0)(3,0.45)
      \readdata[nStep=20]{\graph}{data/gaussian-Kl-H1300.txt}
      \dataplot[plotstyle=bar,barwidth=0.15,linewidth=0.2pt,fillstyle=solid,fillcolor=vlightgray]{\graph}
      \psplot[linewidth=1.5pt,plotpoints=200]{-3.1}{3.1}{0.3989*2.718^(-x^2/2.0)}
    \end{pspicture}}
  \caption{Distribution of sums of trace functions for the (real-valued) Kloosterman sum $\Kl_{2}$ modulo $p=7927$. In bold, the density function of a standard normal random variable.}
\end{figure}

\subsection{Quantitative version}

  Actually, Lamzouri used more precise information than the central limit theorem: the first moments of the model correspond to those of a Gaussian, and are more generally bounded by them. This allows him to approximate the characteristic function of $\left(S(\chi_p,x,H_p)\right)_{x\in\F_p}$ asymptotically, and in turn, gives a bound on the error term for the joint distribution function (what we will call a \textit{quantitative version} of the convergence in law result) by using an identity of Selberg.\\
    
We also get a quantitative version for trace functions by using the fact that moments\footnote{\label{footnote:moments}For a complex-valued random variable $X$, we consider here the moments $\E(X^k\overline{X}^r)$ (and \textit{not} $\E((\Re{X})^k(\Im{X})^r)$); see Remark \ref{rem:momentsconjreim} below.} of traces of random matrices in classical groups are also Gaussian (in $\C\cong\R^2$) as the rank grows, as already remarked and exploited for example by Diaconis-Shahshahani \cite{DiaconisShahshahani}, Pastur-Vasilchuck \cite{PasturVasilchuk}, as well as Larsen \cite{LarsenSatoTateNormal} in the context of trace functions.

More precisely, one rather needs subgaussian bounds on high order moments with respect to the rank, but exploiting the fact that they become exactly Gaussian allows to improve the error terms as the rank grows.

Hence, this uses a different phenomenon than the averaging of the central limit theorem: the random variables modeling the values of the trace function are themselves ``close to Gaussian''.\\

The following is then the extension of the main theorem of \cite{LamzShortSums} (rather than Theorem \ref{thm:gaussian1}):

  \begin{theorem}\label{thm:gaussian2}
  In the notations and hypotheses of Theorem \ref{thm:gaussian1}, fix $\varepsilon\in(0,1/2)$ and let $R$ be the rank of the monodromy group of the family. For any closed rectangle $A\subset\C\cong\R^2$ with sides parallel to the coordinate axes and Lebesgue measure $\mu(A)$, the probability
  \[P\left(\frac{S(t_q,I_q+x)}{\sqrt{|I_q|}}\in A\right)=\frac{|\{x\in\F_q : S(t_q,I_q+x)/\sqrt{|I_q|}\in A\}|}{q}\]
  is given by
\[P(\Nc\in A)+O_\varepsilon\left(\mu(A)\left(q^{-\frac{1}{2}+\varepsilon}+\left(\frac{\log{|I_q|}}{\log{q}}\right)^{\beta}+\frac{1}{\sqrt{|I_q|}}\right)\right)\]
when $q, \ |I_q|\to\infty$ with
\[
\begin{cases}
  \log{|I_q|}=o(\log{q})&:\text{real-valued and Kummer cases}\\
  |I_q|=o\left((\log{q})^{\frac{2R}{(2R-1)(1+\varepsilon)}}\right)&:\text{otherwise},
\end{cases}
\]
where $\Nc$ is a normal random variable in $\C$ with mean $0$ and covariance matrix $\Gamma$ as in Theorem \ref{thm:gaussian1}, and
\[\beta=
\begin{cases}
  1/2-\varepsilon&:\text{real-valued and Kummer cases}\\
  \frac{R-1}{2R-1}&:\text{otherwise}.
\end{cases}
\]
\end{theorem}

\begin{remark}\label{rem:rangeNSD}
  By using a generalization of the Berry-Esseen inequality from \cite{NormAppr}, we improve the method of Lamzouri, which is necessary in the non-real-valued case (see the outline at the beginning of Section \ref{sec:quantitative}). Moreover:
  \begin{enumerate}
  \item In the self-dual case, Theorem \ref{thm:gaussian2} recovers the bound and the range of \cite{LamzShortSums}, with an improvement on the power of $|I_q|$ (from $1/4$ to $1/2$), thanks to a modification of the method.
  \item   In the non-self-dual case, we recover the bound valid for Dirichlet characters when the rank $R\to\infty$, but under the weaker range $|I_q|=o\left((\log{q})^{\frac{R}{R-1}}\right)$ than the one for which Theorem \ref{thm:gaussian1} is valid. We will explain the reason for this later on.
  \end{enumerate}
\end{remark}

\subsection{Examples}\label{subsec:examplesTwisting}

In Section \ref{sec:examples}, we will prove that natural families arising from the examples of Section \ref{subsubsec:examplesTF} are coherent, so that Theorems \ref{thm:gaussian1} and \ref{thm:gaussian2} apply to them.

To make the arithmetic and geometric monodromy groups coincide, we may eventually need to replace a family $(\Fc_q)_q$ by the twisted family $(\alpha_q\otimes\Fc_q)_q$ for $\alpha_q\in\overline\Q_\ell$ a Weil number of weight $0$. This has simply the effect of multiplying the trace function by $\alpha_q$, and the covariance matrix of Theorem \ref{thm:gaussian1} by the orthonormal matrix
\[\left(
  \begin{matrix}
    \Re{\alpha_q}&-\Im{\alpha_q}\\\Im{\alpha_q}&\Re{\alpha_q}
  \end{matrix}
\right),\]
where we identify $\alpha_q$ with its image through the fixed isomorphism $\iota: \overline\Q_\ell\to\C$.

\subsection{Moments of random matrices in classical groups}

  As we mentioned, an important ingredient in the proof of Theorem \ref{thm:gaussian2} is the following:
\begin{proposition}\label{prop:momentsRM}
  For $n\ge 1$, let $G$ be $\SL_{n+1}(\C)$, $\Sp_{2n}(\C)$ or $\SO_{n+1}(\C)$ with standard representation $\Std$. Then, for $R=\rank(G)$ (namely $n$, $n$ and $\floor{(n+1)/2}$ respectively):
  \begin{enumerate}
  \item If $\Std$ is self-dual (i.e. in the symplectic case),
    \begin{align}
      \label{eq:quantcondReal}
      \mult_1(\Std^{\otimes k})&=0 &&(k\ge 0\text{ odd}),\\
      \mult_1(\Std^{\otimes k})&=(k-1)!! &&(0\le k\le R \text{ even}),\label{eq:quantcondReal2}\\
      \mult_1(\Std^{\otimes k})&\le(k-1)!! &&(k\ge 1).\label{eq:quantcondReal3}
    \end{align}
  \item Otherwise,
    \begin{align}
      \label{eq:quantcondComplex}
      \mult_1(\Std^{\otimes k}\otimes D(\Std^{\otimes k}))&=k!  &&(0\le k\le R),\\
      \mult_1(\Std^{\otimes k}\otimes D(\Std^{\otimes r}))&=0&&(0\le k\neq r\le R),\label{eq:quantcondComplex2}\\
      \mult_1(\Std^{\otimes k}\otimes D(\Std^{\otimes r}))&\le\sqrt{k!r!}&&(k,r\ge 0),\label{eq:quantcondComplex3}
    \end{align}
  \end{enumerate}
  where $\mult_1(\,\cdot\,)$ denotes the multiplicity of the trivial representation in a representation of $G$.
\end{proposition}

New aspects compared to existing works are the bounds \eqref{eq:quantcondReal3} and \eqref{eq:quantcondComplex3} that we need on the large order moments with respect to the rank.

\begin{remark}\label{rem:rangeNSD2}
  Recall that (see Section \ref{subsec:charGaussAppr}):
  \begin{itemize}
  \item For $k,r\ge 0$ distinct integers, the $(k,r)$-th moment of a standard Gaussian in $\R^2\cong\C$ is zero.
  \item For $k$ odd, the $k$th moment of a standard Gaussian in $\R$ is zero.
  \end{itemize}
  In the self-dual case, odd moments are zero even for high rank, but in the non-self-dual case, we will see that there are infinitely many nonzero nondiagonal terms. This is the reason for the restricted range in the non-self-dual case of Theorem \ref{thm:gaussian2} noted in Remark \ref{rem:rangeNSD}.
\end{remark}

\section{Probabilistic model}\label{sec:probModel}

We start by setting up a probabilistic model for the random variable $(S(t_q,I_q+x))_{x\in\F_q}$, motivated by Deligne's equidistribution theorem and Lamzouri's work \cite{LamzShortSums} for Dirichlet characters. We then compute its moments.

\subsection{Deligne's equidistribution theorem}

Theorem \ref{thm:DeligneSumTF} and Weyl's equidistribution criterion lead to the following, which shows that there is always an equidistribution result in a coherent family.

\begin{theorem}[Deligne]\label{thm:DeligneED}
  Let us fix an isomorphism $\iota: \overline\Q_\ell\to\C$, and let $(\Fc_q)_q$ be a coherent family of $\ell$-adic sheaves over $\F_q$ with monodromy group $G\le\GL_n(\C)$. Let $K\le G(\C)$ be a maximal compact subgroup.

  For every $x\in U_{\Fc_q}(\F_q)$, the semisimple part of the Jordan-Chevalley decomposition of $\iota\rho_{\Fc_q}(\Frob_{x,q})$ in $G$ gives a well-defined conjugacy class $\theta_{x,q}\in K^\sharp$ such that $t_\Fc(x)=\tr(\theta_{x,q})$.

  When $q\to\infty$, the set $\{\theta_{x,q} :x\in U_\Fc(\F_{q})\}$ becomes equidistributed in $K^\sharp$ with respect to the pushforward of the normalized Haar measure of $K$.
\end{theorem}
\begin{proof}
  This is a variant of \cite[Chapter 3]{KatzGKM} and \cite[Chapter 9]{KatzSarnak91}.
\end{proof}

\subsection{Probabilistic model}

Theorem \ref{thm:DeligneED} suggests to model the random variable
\[\Big(\rho_{\Fc_q}(\Frob_{x,q})\Big)_{x\in U_{\Fc_q}(\F_q),}\]
(with respect to the uniform measure on $\F_q$) as $Y=\pi(X)$, where $X$ is a random variable uniformly distributed in a maximal compact subgroup $K$ of $G$ with respect to the normalized Haar measure and $\pi: K\to K^\sharp$ is the projection to the conjugacy classes.

We shall then accordingly model the random variable
\[\Big(t_{\Fc_q}(x)\Big)_{x\in \F_q}\text{ by }Z=\tr(Y).\]

\begin{remark}
  In \cite{LamzShortSums}, the values of Dirichlet characters of order $d$ are modeled by random variables uniformly distributed in the unit circle, while in our model, by uniform random variables in the roots of unity of order $d$ (the monodromy group of a Kummer sheaf associated to a Dirichlet character of order $d$) are used. Since the moments are the same (see Remark \ref{rem:momentsS1}), this will make no difference.
\end{remark}

\subsubsection{Sums of shifts}

Similarly, for $I\subset\F_q$ of size $H\ge 1$, we will model the random vector
\[\Big(\left(t_{\Fc_q}(x+a)\right)_{a\in I}\Big)_{x\in\F_q}\]
by $(Z_1,\dots,Z_H)$, for $Z_i$ independent distributed like $Z$.\\

Therefore, the sum of shifts
\[\Big(S(t_{\Fc_q},I+x)\Big)_{x\in\F_q}=\left(\sum_{y\in I} t_{\Fc_q}(y+x)\right)_{x\in\F_q}\]
will be modeled by the random walk $S(H)=Z_1+\dots+Z_H$, as in \cite{LamzShortSums}.

\subsection{Computation of the moments}
\begin{proposition}[Probabilistic moments]\label{prop:momentsprob}
  For all integers $k,r\ge 0$ and $H\ge 1$, the moment
  \[M_\prob(k,r;H):=\E(S(H)^k\overline{S(H)^r})\]
  is equal to
  \begin{eqnarray*}
    \sum_{\substack{k_1+\dots+k_H=k\\k_i\ge 0}}\sum_{\substack{r_1+\dots+r_H=r\\r_i\ge 0}}&&\binom{k}{k_1\dots k_H}\binom{r}{r_1\dots r_H}\\[-0.5cm]
  &&\times\prod_{i=1}^H \mult_1(\Std^{\otimes k_i}\otimes D(\Std^{\otimes r_i})),
  \end{eqnarray*}
  where $\Std$ is the standard representation of $G\le\GL_n(\C)$ and $D(\Std)$ its dual.
\end{proposition}
\begin{proof}
  By independence and the multinomial formula, $M_\prob(k,r;H)$ equals
\[\sum_{\substack{k_1+\dots+k_H=k\\k_i\ge 0}}\sum_{\substack{r_1+\dots+r_H=r\\r_i\ge 0}}\binom{k}{k_1\dots k_H}\binom{r}{r_1\dots r_H}\prod_{i=1}^H \E(Z_i^{k_i}\overline{Z_i^{r_i}}).\]
By the Peter-Weyl Theorem,
\begin{eqnarray*}
  \E(Z_i^k\overline{Z_i^r})&=&\int_\C x^k\overline{x^r} d(\tr_*\mu)(x)=\int_{K^\sharp} \tr(g)^k\overline{\tr(g)^r}d\mu(g)\\
  &=&\mult_1(\Std^{\otimes k}\otimes D(\Std^{\otimes r})),
\end{eqnarray*}
where $\mu$ is the normalized Haar measure on $K$, since $\tr$ (resp. $\overline{\tr}$) is the character associated to the standard representation of $G$ (resp. its dual).
\end{proof}

\begin{remark}\label{rem:momentsconjreim}
  The covariance matrix \eqref{eq:covmatrix} of Theorem \ref{thm:gaussian1} is given with respect to the standard basis $1,i$ of $\C$ as $\R$-vector space, and a nice feature of the result is that the matrix is diagonal, i.e. the real and imaginary parts are independent. However, it will be more natural for the proof to make the linear transformation $\left(
    \begin{smallmatrix}
      Z\\\overline{Z}
    \end{smallmatrix}
\right)=\left(
    \begin{smallmatrix}
      1&i\\1&-i
    \end{smallmatrix}
  \right) \left(
    \begin{smallmatrix}
      \Re{Z}\\\Im{Z}
    \end{smallmatrix}
\right)$ and consider as in Proposition \ref{prop:momentsprob} the moments $\E(Z^k\overline{Z}^r)$ instead of $\E((\Re{Z})^k(\Im{Z})^r)$. The reason is that conjugation has the algebraic interpretation of dualization of representations, characters, and trace functions. In the real-valued case, there is no difference.
\end{remark}

\begin{lemma}\label{lemma:covXi}
  We have $\E(Z)=0$ and the covariance matrix of the random vector $Z=(\Re{Z},\Im{Z})$ is
  \[\left(\begin{matrix}
    1&0\\0&0
  \end{matrix}\right)\text{if }\Std\text{ is self-dual, }
\frac{1}{2}\left(\begin{matrix}
    1&0\\0&1
  \end{matrix}\right)\text{otherwise}.\]
\end{lemma}
\begin{proof}
  Since the sheaf is geometrically irreducible, $\Std$ is irreducible, so that $\E(Z)=\mult_1(\Std)=0$ by Schur's Lemma. Moreover, for every integer $r\ge 0$ we have
  \[\mult_1(\Std^{\otimes r})=\int_K\tr(g)^rd\mu(g)=\int_K \overline{\tr(g)^r}d\mu(g)=\mult_1(D(\Std)^{\otimes r})\]
  where the second equality follows from the fact that $\mult_1(\Std^{\otimes r})$ is an integer.
  Using this, we find that the covariance matrix of $Z$ is
\[\frac{1}{2}\left(
      \begin{matrix}
        \mult_1(\Std^{\otimes 2})+1&0\\
        0&1-\mult_1(\Std^{\otimes 2})
      \end{matrix}
    \right).\]
    Finally, again by Schur's Lemma, $\mult_1(\Std^{\otimes 2})=\mult_1(\Std\otimes D(D(\Std))=\delta_{\Std\text{ self dual}}$.
\end{proof}

\begin{lemma}
    Let $\Fc$ be a geometrically irreducible $\ell$-adic sheaf over $\F_q$, pointwise pure of weight $0$, with monodromy groups $G=G_\geom(\Fc)=G_\arith(\Fc)\le\GL_n(\C)$. The following are equivalent:
  \begin{enumerate}
  \item\label{item:selfDualReal1} For any finite extension $\F_{q'}/\F_q$, the trace function $t: \F_{q'}\to\C$ is real-valued.
  \item\label{item:selfDualReal2} The standard representation of $G\le\GL_n(\C)$ is self-dual.
  \item\label{item:selfDualReal3} $\mult_1(\Std^{\otimes 2})=\mult_1(\Std\otimes D(\Std))=1$.
  \end{enumerate}
\end{lemma}
\begin{proof}
   By a result of Deligne \cite[1.3.9]{Del2} (see \cite[9.0.12]{KatzSarnak91}), we have $G^0$ semisimple, so that there is an equivalence of categories between representations of the algebraic group $G$, of the Lie group $G(\C)$, or of $K$. Note that by assumption, $\Std$ is irreducible. By the Chebotarev density theorem, the Frobenius conjugacy classes $\Frob_{x,q'}$, for $\F_{q'}/\F_q$ a finite extension and $x\in U_\Fc(\F_{q'})$, are dense in $\pi_{1,q}$ (see \cite[I.2.2, Corollary 2 a)]{Serre89}). Thus, \ref{item:selfDualReal1} is indeed equivalent to having $\iota(\tr(\rho_\Fc(\pi_{1,q})))\subset\R$ for all $q$, which in turn holds if and only if $\tr(G)\subset\R$. Hence, \ref{item:selfDualReal1} is equivalent to \ref{item:selfDualReal2} by character theory of $G(\C)$. If \ref{item:selfDualReal2} holds, then \[\mult_1(\Std^{\otimes 2})=\mult_1(\Std\otimes D(\Std))=1\]
  by Schur's Lemma, so that \ref{item:selfDualReal3} holds. If \ref{item:selfDualReal3} holds, we have
  \[1=\int_K \tr(g)^2dg=\left|\int_K\tr(g)^2dg\right|\le\int_K |\tr(g)|^2dg=1,\]
  so that $\tr(g)^2=|\tr(g)|^2$ for almost all $g\in K$. Hence, $\tr(g)\in\R$ almost everywhere in $K$, and this holds everywhere in $K$ since a nonempty open set has positive Haar measure. Thus \ref{item:selfDualReal1} follows by the first statement in Theorem \ref{thm:DeligneED}.
\end{proof}

Hence, we conclude by the two preceding Lemmas that the covariance matrix of $Z$ is equal to that given in \eqref{eq:covmatrix}.

\section{Qualitative version (Theorem \ref{thm:gaussian1})}\label{sec:qualitative}

\subsection{Strategy and comparison with other approaches}

The idea of the proof of Theorem \ref{thm:gaussian1} is the following:
\begin{enumerate}
\item By the method of moments, it suffices to show that the moments of the random variable \eqref{eq:RVSfI} tend to that of the Gaussian $\Nc$.
\item\label{item:proofGaussian1Step2} We show that the probabilistic model of Section \ref{sec:probModel} is accurate, in the sense that the moments of \eqref{eq:RVSfI} converge to that of the model.
\item To conclude, it suffices to apply the central limit theorem (with convergence of moments) to the model.
\end{enumerate}

This is to be compared with the approaches of earlier works which do not use the central limit theorem:

\begin{itemize}
\item Davenport-Erd\H{o}s \cite{DavenportErdos} and Mak-Zaharescu \cite{MakZahShortSums} directly show that the moments of \eqref{eq:RVSfI} are asymptotically Gaussian and apply the method of moments.
\item Lamzouri \cite{LamzShortSums} first proves that his probabilistic model is accurate as in step \ref{item:proofGaussian1Step2} above. He then remarks that the random variable $X$ modeling the values of the Dirichlet characters itself has moments bounded by those of a Gaussian. That allows to approximate the characteristic function of the model for the sums by that of a Gaussian. By using a method of Selberg, this finally gives an approximation for the joint characteristic function. We will comment more on this approach in Section \ref{sec:quantitative}.
\end{itemize}

We shall see that with the $\ell$-adic formalism, the proof that the model is accurate becomes very natural and does not involve explicit computations of moments.

\subsection{Accuracy of the model}

Under the hypotheses and notations of Theorem \ref{thm:gaussian1}, as in Proposition \ref{prop:momentsprob} the moment
\[M_q(k,r;I_q):=\E\left(S(t_q,I_q+x)^k\overline{S(t_q,I_q+x)^r}\right)\]
equals
\begin{eqnarray}
  \sum_{\substack{k_1+\dots+k_H=k\\k_i\ge 0}}\sum_{\substack{r_1+\dots+r_H=r\\r_i\ge 0}}&&\binom{k}{k_1\dots k_H}\binom{r}{r_1\dots r_H}\label{eq:accuracyModel}\\
  &&\aver{x}{\F_q}{q} \prod_{i=1}^{H} t_q(x+a_i)^{k_{i}}\closure{t_q(x+a_i)^{r_{i}}}\nonumber
\end{eqnarray}
for all integers $k,r\ge 0$, where $I_q=\{a_1,\dots,a_H\}$. Thus, by Proposition \ref{prop:momentsprob}, we need to compare ``sums of products'' of trace functions
\begin{equation}
  \label{eq:sumproducts}
  \aver{x}{\F_q}{q} \prod_{i=1}^{H} t_q(x+a_i)^{k_{i}}\closure{t_q(x+a_i)^{r_{i}}}  
\end{equation}
with products of the form
\begin{equation}
  \label{eq:momentsmult}
  \prod_{i=1}^H \mult_1(\Std^{\otimes k_i}\otimes D(\Std^{\otimes r_i}))
\end{equation}
when $k_i,r_i\ge 0$ are integers.

\subsubsection{Sums of products of trace functions}

The estimation of sums of the form $\sum_{x\in\F_q} \prod_{i=1}^H t_i(x)$ for $t_i$ a trace function over $\F_q$ is precisely the question that is surveyed in \cite{FKMSumProducts}, and the link between \eqref{eq:sumproducts} and \eqref{eq:momentsmult} can be made clear through a cohomological interpretation of the sum via Theorem \ref{thm:DeligneSumTF}:
\begin{itemize}[leftmargin=*]
\item For $\bs k,\bs r\in\N^H$, let us consider the sheaf
  \[\Gc^{\bs k,\bs r}=\bigotimes_{1\le i\le H}\left( [+a_i]^*\Fc_q^{\otimes k_i}\otimes D([+a_i]^*\Fc_q)^{\otimes r_i}\right).\]
  By Theorem \ref{thm:DeligneSumTF}, under the hypotheses of Theorem \ref{thm:gaussian1}, \eqref{eq:sumproducts} is equal to
  \[\tr\left(\Frob_q\mid\Gc^{\bs k,\bs r}_{\pi_{1,q}^\geom}\right)+O(r^{S(\bs k,\bs r)}S(\bs k,\bs r)c^2q^{-1/2})\]
  where the implicit constant is absolute, $r=\rank(\Fc_q)$, $c=\cond(\Fc_q)$, and $S(\bs k,\bs r)=\sum_{i=1}^H (k_i+r_i)$.
\item By the Goursat-Kolchin-Ribet criterion of Katz, if $\Fc_q$ is part of a coherent family in the classical case, then the arithmetic and geometric monodromy groups of
  \[\Gc=\bigoplus_{1\le i\le H}[+a_i]^*\Fc_q\]
  coincide and are as large as possible, i.e. isomorphic to $G^H$.
\item Thus, by Remark \ref{rem:DeligneSumTFMon},
\begin{eqnarray*}
  \tr\left(\Frob_q\mid(\Gc_{\bs k,\bs r})_{\pi_{1,q}^\geom}\right)&=&\dim\Lambda_{G_\geom(\Gc)}=\dim  \Lambda_{G^H}\\
  &=&\dim \bigboxtimes_{1\le i\le H} \left(\Std^{\otimes k_i}\otimes D(\Std)^{\otimes r_i}\right)_G\\
  &=&\prod_{1\le i\le H}\mult_1(\Std^{\otimes k_i}\otimes D(\Std)^{\otimes r_i}),
\end{eqnarray*}
for the $G^H$-representation $\Lambda=\bigboxtimes_{1\le i\le H} \left(\Std^{\otimes k_i}\otimes D(\Std)^{\otimes r_i}\right)$.
\end{itemize}

\begin{remark}
  As mentioned in Remark \ref{rem:conductorGrowing}, the conductor of $\Gc^{\bs k,\bs r}$ is unbounded as $H\to\infty$, so that, in contrast with \cite{FKMSumProducts}, we have to keep track of the dependency with the conductors in the error terms.
\end{remark}

Finally, we get:
  \begin{proposition}\label{prop:sumProducts}
  Let $(\Fc_q)_q$ be a coherent family of sheaves over $\F_q$, with monodromy group $G\le\GL_n(\C)$. Let $a_1,\dots,a_H\in\F_q$ be distinct. If $\Fc_q$ is $\{a_1,\dots,a_H\}$-compatible, then for all $\bs k,\bs r\in\N^H$,
  \begin{eqnarray*}
    \frac{1}{q}\sum_{x\in\F_q} \prod_{1\le i\le H} t_{\Fc_q}(x+a_i)^{k_i}\overline{t_{\Fc_q}(x+a_i)}^{r_i}&=&\prod_{1\le i\le H}\mult_1(\Std^{\otimes k_i}\otimes D(\Std)^{\otimes r_i})\\
    &&+O\left(r^{S(\bs k,\bs r)}S(\bs k,\bs r)q^{-1/2}\right)
  \end{eqnarray*}
  where the implicit constant does not depend on $q$, and $S(\bs k,\bs r)=\sum_{i=1}^H (k_i+r_i)$.
\end{proposition}
\begin{proof}
  The proof in the case of a classical monodromy group was sketched above; details can be found in \cite{FKMSumProducts} or \cite{PG16}. For a Kummer sheaf $\Lc_{\chi(f)}$ with $\chi: \F_q^\times\to\C$ of order $d$ and $f\in\F_q(T)$, the sum is by multiplicativity equal to
  \[\frac{1}{q}\sum_{x\in\F_q}\chi(g(x))=\frac{1}{q}\sum_{x\in\F_q}t_{\Lc_\chi(g)},\]
  where $g(X)=\prod_{1\le i\le H} f(X+a_i)^{k_i-r_i}$. Writing $f=f_1/f_2$ and $g=g_1/g_2$ with $f_i,g_i\in\F_q[X]$, we see that $\deg(g_1)+\deg(g_2)\le S(\bs k,\bs r)(\deg(f_1)+\deg(f_2))\le S(\bs k,\bs r)\cond(\Fc_q)$. By Theorem \ref{thm:DeligneSumTF} and Remark \ref{rem:DeligneSumTFMon} applied to the Kummer sheaf $\Lc_{\chi(g)}$, it follows that
  \[\frac{1}{q}\sum_{x\in\F_q}\chi(g(x))=\delta_{g \text{ is a }d-\text{power}}+O(S(\bs k,\bs r)c^2q^{-1/2}).\]
  Observe that $\mult_1(\Std^{\otimes k_i}\otimes D(\Std)^{\otimes r_i})=\delta_{d\mid k_i-r_i}$, so that the claim is clear if $f=X$. Otherwise, the compatibility assumption shows that\footnote{This idea appears on page 9 of the published version of \cite{LamzModm}.} there exists a zero $x$ of $f$ such that $f(x+a)\neq 0$ for all $a\in I$. Indeed, otherwise, for any zero $x$ of $f$ and any integer $d_f\ge 0$, there would exist $a_1,\dots,a_{d_f}\in\overline\F_q$ with $x+a_1,\dots,x+\sum_{i=1}^{d_f} a_i$ distinct zeros of $f$, which is impossible. This implies that $g$ cannot be a $d$-power if $d\nmid k_i-r_i$ for some $i$.
\end{proof}

\subsubsection{Conclusion}

The asymptotic accuracy of the model then follows from Proposition \ref{prop:sumProducts} applied to \eqref{eq:accuracyModel}, recalling that $\sum_{k_1+\dots+k_H=k, \ k_i\ge 0}\binom{k}{k_1\dots k_H}=H^k$:
\begin{proposition}\label{prop:moments} Under the hypotheses of Theorem \ref{thm:gaussian1}, for all integers $k,r\ge 0$ and $I_q\subset\F_q$ of size $H$, we have
  \begin{equation*}
  M_q(k,r;I_q)=M_\prob(k,r;H)+O\left(c^{3(k+r)}q^{-1/2}H^{k+r}\right)
\end{equation*}
with $c=\max_{q} \cond(\Fc_q)$.
\end{proposition}
We make the normalizations
\[\tilde S(t_q,I_q+x)=S(t_q,I_q+x)/|I_q|^{1/2}\text{ and }\tilde S(H)=S(H)/H^{1/2},\]
and for $k,r\ge 0$ we denote by $\tilde M_q(k,r;I_q)$, $\tilde M_\prob(k,r;H)$ the corresponding moments, so that Proposition \ref{prop:moments} becomes:
\begin{equation}
  \label{eq:MpMprobtilde}
  \tilde M_q(k,r;I_q)=\tilde M_\prob(k,r;H)+O\left(c^{3(k+r)}q^{-1/2}H^{\frac{k+r}{2}}\right).
\end{equation}

\subsection{Central limit theorem}

\begin{proposition}\label{prop:CLT}
  Under the hypotheses and notations of Theorem \ref{thm:gaussian1}, the random variable $\tilde S(H)$ converges in law to the random variable $\Nc$ when $H\to\infty$. Moreover,
  \[\lim_{H\to\infty}\tilde M_\prob(k,r;H)=M_\Nc(k,r),\]
  for all integers $k,r\ge 0$, where $M_\Nc(k,r)$ are the moments of $\Nc$.
\end{proposition}
\begin{proof}
  This follows from the two-dimensional Central limit theorem and Lemma \ref{lemma:covXi}. To obtain the convergence of moments, it suffices to show that $\tilde S(H)$ is uniformly integrable (see e.g. \cite[Chapter 5.5]{Gut}), which follows from  \cite[Theorem 7.5.1]{Gut}.
\end{proof}
By \eqref{eq:MpMprobtilde}, this immediately implies:
\begin{corollary}[Moments are asymptotically Gaussian]\label{cor:convMoments}
  Under the hypotheses and notations of Theorem \ref{thm:gaussian1}, we have for all integers $k,r\ge 0$ that
  \[\lim_{q, |I_q|\to\infty} \tilde M_q(k,r;I_q)=M_\Nc(k,r).\]
\end{corollary}

\subsection{Method of moments and proof of Theorem \ref{thm:gaussian1}}

 To conclude the proof of Theorem \ref{thm:gaussian1}, it now suffices to apply the method of moments:
\begin{proposition}[Method of moments for complex-valued random variables]\label{prop:methodMoments}
  Let $(X_n)_{n\ge 0}$ be a sequence of complex random variables with moments $M_{X_n}(k,r)$. If $\lim_{n\to\infty} M_{X_n}(k,r)=M_{X_0}(k,r)$
  for all integers $k,r\ge 0$ and if
  \[\limsup_{k+r\to\infty} \frac{|M_{X_0}(k,r)|^{\frac{1}{k+r}}}{k+r}<\infty,\]
  then $X_n$ converges in law to $X_0$.
\end{proposition}
\begin{proof}
  See for example \cite[Chapter 5.8.4]{Gut}.
\end{proof}
\begin{corollary}[Method of moments for normal convergence]\label{cor:methodMomentsNormal}
  Let $(X_n)_{n\ge 0}$ be a sequence of complex random variables. If for all integers $k,r\ge 0$, the moment $M_{X_n}(k,r)$ converges to the corresponding moment of a normal random variable $\Nc$ as $n\to\infty$, then $X_n$ converges in law to $\Nc$.
\end{corollary}

Hence, by Corollary \ref{cor:methodMomentsNormal}, Theorem \ref{thm:gaussian1} follows directly from Corollary \ref{cor:convMoments}.

\section{Quantitative version (Theorem \ref{thm:gaussian2})}\label{sec:quantitative}

\subsection{Review of Lamzouri's method}
  We recall that the idea of \cite{LamzShortSums} is to remark that the random variable $Z$ modeling the values of the Dirichlet characters has moments bounded by those of a Gaussian. In particular, this implies that if $S(H)=Z_1+\dots+Z_H$ with $Z_i$ independent distributed like $Z$ as above, we have
  \[\E\left((\Re{S(H)})^{2k}(\Im{S(H)})^{2r}\right)\ll (k+r)!H^{k+r}\]
  (see \cite[(3.5)]{LamzShortSums}), which is a square-root cancellation over the trivial bound $H^{2(k+r)}(2k+2r)!$. This implies that one can:
  \begin{enumerate}
  \item Approximate the characteristic function of $(S(\chi_p,x,H_p))_{x\in\F_p}$ asymptotically by that of the probabilistic model when $p,H_p\to\infty$ (see the proof of \cite[Theorem 3.1]{LamzShortSums}).
  \end{enumerate}
  Lamzouri then proceeds as follows:
\begin{enumerate}[start=2]
\item\label{item:proofTheoremGaussian2Step1} As in the classical proof of the central limit theorem, the characteristic function of the model is approximated by that of a Gaussian (\cite[Lemma 3.2]{LamzShortSums}).
\item\label{item:proofTheoremGaussian2Step2} Combining the last two points, this gives an asymptotic approximation of the characteristic function of $(S(\chi_p,x,H_p))_{x\in\F_p}$ by that of a Gaussian (\cite[Theorem 3.1]{LamzShortSums}).
\item\label{item:proofTheoremGaussian2Step3} Using a smooth approximation for the sign function involving characteristic functions, due to Selberg (\cite[(4.4)]{LamzShortSums}), one gets an approximate expression of the joint distribution function from the characteristic function, which allows to conclude.
\end{enumerate}
\subsection{Generalization to trace functions}
As we explained in the introduction, this can be generalized to coherent families of trace functions thanks to Proposition \ref{prop:momentsRM}.
We will however proceed a bit differently than Lamzouri, skipping steps \ref{item:proofTheoremGaussian2Step1}--\ref{item:proofTheoremGaussian2Step2} above and:
\begin{enumerate}
\item Directly use step \ref{item:proofTheoremGaussian2Step3} to approximate the joint distribution function of the random variable \eqref{eq:RVSfI} by that of the model.
\item Apply a generalization to higher dimensions of the Berry-Esseen theorem appearing in \cite{NormAppr}, to obtain an approximation of the joint distribution function of the model.
\end{enumerate}

\subsection{Characteristic function of a Gaussian}\label{subsec:charGaussAppr}
Let us recall that if $Z$ is a normal random variable in $\R$ with mean $0$ and variance $\sigma^2$, the moments are
\[\E(Z^k)=\begin{cases}
      0&\text{ if }k\ge 1\text{ is odd}\\
      \sigma^k(k-1)!!&\text{ if }k\ge 0\text{ is even}
    \end{cases}\]
and its characteristic function is $u\mapsto\E(e^{iuZ})=e^{-\frac{1}{2}\sigma^2u^2}$.

Hence, if $Z$ is a normal random variable in $\C\cong\R^2$ with mean $0$ and diagonal covariance matrix $\sigma^2\left(\begin{smallmatrix}1&0\\0&1\end{smallmatrix}\right)$, then its characteristic function is
\[(u,v)\mapsto\tilde\phi(u,v)=\E\left(e^{i(u\Re{Z}+v\Im{Z})}\right)=e^{-\frac{\sigma^2}{2}(u^2+v^2)}.\]

As we explained in Remark \ref{rem:momentsconjreim}, we will continue to rather work with moments of the form $\E(Z^k\overline{Z}^r)$ and characteristics functions of the form $(u,v)\mapsto\phi(u,v)=\E(e^{i(uZ+v\overline{Z})})$, which renders the computations easier and more natural in our setting. Note that
\begin{eqnarray}
  \phi(u,v)&=&\tilde\phi(u+v,i(u-v))\text{ and}\nonumber\\
  \tilde\phi(u,v)&=&\phi\left(\frac{u-iv}{2},\frac{u+iv}{2}\right)\label{eq:phitildephi}
\end{eqnarray}
for all $u,v\in\C$. Hence, $\phi(u,v)=e^{-2\sigma^2uv}$, so that $\E\big(Z^k\overline{Z}^r\big)=(2\sigma^2)^kk!\delta_{k=r}$.

\subsection{Approximation of characteristic functions through moments}
\begin{lemma}\label{lemma:charfunctMoments}
  Let $X_1,X_2$ be complex random variables with moments $M_j(k,r)=\E(X_j^k\overline{X_j}^r)$ and characteristic functions $(u,v)\mapsto\phi_j(u,v)=\E(e^{i(uX_j+v\overline{X_j})})$ \textup{(}$j=1,2$\textup{)} for $u,v\in\C$ and $k,r\ge 0$ integers. Assume that
  \[M_1(k,r)=M_2(k,r)+O(g(k,r))\]
  for all $k,r\ge 0$ with some $g:\N^2\to\R$. Then for any fixed even integer $N\ge 1$ and $u\in\C$, we have
  \begin{eqnarray*}
    \phi_1(u,\overline{u})&=&\phi_2(u,\overline{u})+O\left(\frac{|u|^{N}}{N!}(|M_1(N/2,N/2)|+|M_2(N/2,N/2)|)\right)\\
    &&+O\left(\sum_{n<N} \frac{|u|^n}{n!}\sum_{a=0}^n \binom{n}{a}|g(a,n-a)|\right).
  \end{eqnarray*}
In particular, if $g(k,r)=h(k+r)$ for all $k,r\ge 0$ for some $h:\N\to\R$, we have
  \begin{eqnarray*}
    \phi_1(u,\overline{u})&=&\phi_2(u,\overline{u})+O\left(\frac{|u|^{N}}{N!}(|M_1(N/2,N/2)|+|M_2(N/2,N/2)|)\right)\\
    &&+O\left(\max_{n<N} |h(n)| (1+|u|^N)\right).
  \end{eqnarray*}
  If $X_1$, $X_2$ are random variables in $\R$, then a similar relation holds for $\phi_1(u,0)$ and $\phi_2(u,0)$ with $u\in\R$.
\end{lemma}
\begin{proof}
  It suffices to use the expansion $e^{ix}=\sum_{n<N} \frac{i^nx^n}{n!}+O\left(\frac{|x|^N}{N!}\right)$ valid for $x\in\R$.
\end{proof}

\subsection{Bounding moments}
In order to apply Lemma \ref{lemma:charfunctMoments}, we will need bounds on the moments $M_\prob(N,N; H)$, provided by Proposition \ref{prop:momentsRM}. Recall that by Proposition \ref{prop:momentsprob}, we have
\[M_\prob(N,N; H)=N!^2\sum_{\substack{k_1+\dots+k_H=N\\k_i\ge 0}}\sum_{\substack{r_1+\dots+r_H=N\\r_i\ge 0}}\prod_{i=1}^H \frac{\E(Z_i^{k_i}\overline{Z_i^{r_i}})}{k_i!r_i!}.\]
Note that if all $Z_i$ were normal variables in $\C$ with mean $0$ and covariance matrix $\sigma^2
\left(\begin{smallmatrix}1&0\\0&1\end{smallmatrix}\right)$ (resp. $\sigma^2
\left(\begin{smallmatrix}1&0\\0&0\end{smallmatrix}\right)$), then this would be equal to $(2\sigma^2)^NN!H^N$ (resp. $\sigma^{2N}(2N-1)!!H^N$).
\begin{proposition}[Non-self-dual case]\label{prop:boundMprobNSD}
  If the conclusions of Proposition \ref{prop:momentsRM} hold, then in the non-self-dual case, $M_\prob(N,N;H)\le (N+H-1)^NH^N$.
\end{proposition}
\begin{proof}
  By the Cauchy-Schwarz inequality,
  \[M_\prob(N,N;H)\le \left(\sum_{\substack{k_1+\dots+k_H=N\\k_i\ge 0}}\frac{N!}{\sqrt{k_1!\dots k_H!}}\right)^2\le H^N \frac{(N+H-1)!}{(H-1)!},\]
  since the number of weak $H$-compositions\footnote{Recall that a \textit{weak $H$-composition} of an integer $N$ is a tuple of nonnegative integers $(k_1,\dots,k_H)$ such that $k_1+\dots+k_H=N$.} of $N$ is equal to $\binom{N+H-1}{H-1}$. Finally, we use that $\frac{(N+H-1)!}{(H-1)!}\le(N+H-1)^N$.
\end{proof}

\begin{remark}
  In Remarks \ref{rem:rangeNSD} and \ref{rem:rangeNSD2}, we explained that the reason for the restriction on the range in Theorem \ref{thm:gaussian2} for the non-self-dual case came from the fact that $X_i$ may have infinitely many nonzero nondiagonal moments. If \eqref{eq:quantcondComplex2} in Theorem \ref{thm:gaussian2} held for all distinct $k,r\ge 0$, then we would get the bound $H^N$ instead of $H^N(N+H-1)^N$. We will see later how this additional exponential in $H$ modifies the aforementioned range.
\end{remark}

For Dirichlet characters, we can achieve the following better bound:
\begin{proposition}[Non-self-dual case, Kummer sheaves]\label{prop:boundMprobNSDChar}
  In the case of Kummer sheaves, we have $M_\prob(N,N;H)\le N!H^N$.
\end{proposition}
\begin{proof}
  If $Z$ is a random variable uniformly distributed in $\mu_d(\C)$, then
  \[\E(Z^k\overline{Z}^r)=\frac{1}{d}\sum_{i=0}^{d-1}\zeta_d^{i(k-r)}=\delta_{k=r},\text{ so}\]
  \[M_\prob(N,N;H)\le N!\sum_{\substack{k_1+\dots+k_H=N\\k_i\ge 0}} \frac{N!}{(k_1!\dots k_H!)^2}\le N!H^N.\]
\end{proof}
\begin{remark}\label{rem:momentsS1}
  Actually, Lamzouri \cite{LamzShortSums} models $Z$ as a random vector uniformly distributed on the unit circle $S^1$. This is equivalent since the moments are then
  \[\E(Z^k\overline{Z}^r)=\frac{1}{2\pi}\int_{0}^{2\pi} e^{i\theta(k-r)}d\theta=\delta_{k=r} \hspace{0.2cm} (k,r\ge 0).\]
\end{remark}

\begin{proposition}[Self-dual case]\label{prop:boundMprobSD}
  If the conclusions of Proposition \ref{prop:momentsRM} hold, then in the self-dual case,
  \[M_\prob(N,N;H)\le (2N-1)!!H^N.\]
\end{proposition}
\begin{proof}
  Since $(k-1)!!=\frac{k!}{2^{k/2}(k/2)!}$ for $k\ge 1$ odd,
  \begin{eqnarray*}
    M_\prob(N,N;H)&\le&\sum_{\substack{k_1+\dots+k_H=2N\\ k_i\ge 0\text{ even}}} \frac{(2N)!}{k_1!\dots k_H!}\prod_{i=1}^H \frac{k_i!}{2^{k_i/2}(k_i/2)!}\\
    &=&\frac{(2N)!}{N!2^N}\sum_{\substack{m_1+\dots+m_H=N\\ m_i\ge 0}}\binom{N/2}{m_1\dots m_H}=(2N-1)!!H^N.
  \end{eqnarray*}
\end{proof}
\subsection{Approximation of joint distribution functions through characteristic functions}
The following result appears in \cite{LamzShortSums}, and follows from a smooth approximation of the sign function (and thus of the characteristic function of a rectangle in $\R^2$) by Selberg in \cite{Selberg92}.

\begin{proposition}\label{prop:approximationSelberg}
  Let $X$ be a complex random variable with characteristic function $\phi_X(u,v)=\E\left(e^{i(u\Re X+v\Im X)}\right)$ $(u,v\in\R)$ and $A=[a,b]\times[c,d]$ be a rectangle in $\R^2\cong\C$. Then, for any real number $t>0$,
\begin{eqnarray*}
  P(X\in A)&=&\frac{1}{2}\Re\int_0^t\int_0^t G(u/t)G(v/t)\Big( \phi_X(2\pi u,-2\pi v)f_{a,b}(u)\overline{f_{c,d}(v)}\\
  && \hspace{2cm}-\phi_X(2\pi u,2\pi v)f_{a,b}(u)f_{c,d}(v) \Big)\frac{du}{u}\frac{dv}{v}\\
  &&+O\left(\frac{1}{t}\int_0^t (|\phi_X(2\pi u,0)|+|\phi_X(0,2\pi u)|)du\right)
\end{eqnarray*}
where $G(u)=\frac{2u}{\pi}+2(1-u)\cot(\pi u)$ for $u\in[0,1]$ and $f_{\alpha,\beta}(u)=(e(-\alpha u)-e(-\beta u))/2$ for $u\in\C$, $\alpha,\beta\in\R$.
\end{proposition}
\begin{proof}
  See \cite[Section 4]{LamzShortSums}.
\end{proof}
\begin{corollary}\label{cor:approximationSelberg}
  If $X,Y$ are complex random variables such that there exists a nonnegative continuous function $g:\R^2\to\R_{\ge 0}$ with
  \[\phi_{X}(2\pi u,2\pi v)=\phi_Y(2\pi u,2\pi v)+O(g(|u|,|v|))\]
  for all $u,v\in\R$, then we have
  \begin{eqnarray*}
  P(X\in A)&=&P(Y\in A)\\
           &&+O\left(\int_0^t\int_0^t g(u,v)dudv+\frac{1}{t}\int_0^t (g(u,0)+g(0,u))du\right)\\
           &&+O\left(\frac{1}{t}\int_0^t (|\phi_X(2\pi u,0)|+|\phi_X(0,2\pi u)|)du\right).
\end{eqnarray*}
\end{corollary}
\subsection{Central limit theorem and sums of quasi-normal random variables}
\begin{lemma}\label{lemma:charfunctMomentsAverage}
  For $H\ge 1$, let $X_1,\dots,X_H$ be independent identically distributed random variables and consider
  \[S(H)=\frac{X_1+\dots+X_H}{\sqrt{H}}.\]
  Assume that for $0\le k,r\le N$, the moments $M(k,r)=\E(X_1^k\overline{X_1}^r)$ of $X_1$ correspond to the moments of a normal random variable in $\C$ with mean $0$ and covariance matrix $\sigma^2\left(\begin{smallmatrix}1&0\\0&1\end{smallmatrix}\right)$, respectively $\left(\begin{smallmatrix}\sigma^2&0\\0&0\end{smallmatrix}\right)$. Then the characteristic function $\phi_H(u,v)=\E\left(e^{i(uS(H)+v\overline{S(H)})}\right)$ of $S(H)$ satisfies
  \[\phi_H(u,\overline{u})=e^{-2\sigma^2|u|^2}\left(1+O\left(\frac{|u|^N}{H^{(N-1)/2}}\right)\right),\]
  when $u\in\C$ with $|u|\le H^{\frac{N-2}{2N}}$, respectively
\[\phi_H(u,0)=e^{-\frac{1}{2}\sigma^2u^2}\left(1+O\left(\frac{|u|^N}{H^{(N-1)/2}}\right)\right)\]
when $u\in\R$ with $|u|\le H^{\frac{N-2}{2N}}$.
\end{lemma}
\begin{proof}
  By independence of the $X_i$, we have $\phi_H(u,v)=\phi\left(\frac{u}{\sqrt{H}},\frac{v}{\sqrt{H}}\right)^H$ where $\phi(u,v)=\E\left(e^{i(uX_1+v\overline{X_1})}\right)$ is the characteristic function of $X_1$. Then
  \begin{eqnarray*}
    \phi_H(u,\overline{u})&=&\left(e^{-2\sigma^2|u|^2/H}+O\left(\frac{|u|^N}{H^{(N+1)/2}}\right)\right)^H\\
    &=&e^{-2\sigma^2|u|^2}\left(1+O\left(\frac{|u|^N}{H^{(N-1)/2}}\right)\right)
  \end{eqnarray*}
  in the first case, since $(e^a+O(a^2))^H=e^{aH}(1+O(a^2H))$ if $a^2H\le 1$. The second case is similar.
\end{proof}
\subsection{Normal approximation}
Below, we give a particular case of the generalization of the Berry-Esseen Theorem in higher dimensions appearing in \cite{NormAppr}.

\begin{proposition}\label{prop:BerryEsseen}
  Let $X_1,\dots,X_H$ be independent and identically distributed random vectors in $\R^2$, satisfying
  \[\E(X_1)=0\text{, and } \E(||X_1||^4)<\infty,\]
  and let $S(H)=\frac{X_1+\dots+X_H}{\sqrt{H}}$. Then for any $A\subset \R^2$ Borel-measurable,
  \begin{equation*}
    P(S(H)\in A)=P(\Nc\in A)+O(\mu(A)H^{-1/2}),
  \end{equation*}
  where $\Nc$ is a normal random vector in $\R^2$ with mean $0$ and covariance $\Cov(X_1)$.
\end{proposition}
\begin{proof}
  This follows from \cite[Theorem 13.2]{NormAppr} taking $d=2$ and $f=1_A$. Note that, under the notations of the latter,
  \[\delta_H\ll \frac{H\log{H}}{e^{CH}} \text{ and } \omega_f^*(2^{7/2}\pi^{-1/3}2^{4/3}\rho_3H^{-1/2}:\Phi)\ll \frac{\mu(A)}{\sqrt{H}}\]
for some absolute constant $C>0$. Thus, for $\Phi$ the density function of $\Nc$,
\begin{eqnarray*}
  \left|\int_A  d(S(H)-\Phi)\right|&\ll&\omega_f(\R^2)\left(\frac{1}{\sqrt{H}}+\frac{\log{H}}{H}+\frac{1}{H\sqrt{\log{H}}}+\frac{1}{e^{CH}\sqrt{H\log{H}}}\right)\nonumber\\
    &&+\omega_f^*(2^{7/2}\pi^{-1/3}2^{4/3}\rho_3N^{-1/2}:\Phi)\nonumber\\
    &\ll&\mu(A)H^{-1/2}.
\end{eqnarray*}
\end{proof}
\subsection{Proof of Theorem \ref{thm:gaussian2}}
Combining the above results, we can finally prove Theorem \ref{thm:gaussian2}, conditionally on Proposition \ref{prop:momentsRM}. Let us consider the characteristic functions
\[\phi_{q,I}(u,v)=\E\left(e^{i(u\tilde S(t_q,I+x)+v\overline{\tilde S(t_q,I+x)})}\right) \ (u,v\in\C)\]
of the normalized complex-valued random variable $(\tilde S(t_q,I+x))_{x\in\F_q}$ and
\[\phi_H(u,v)=\E\left(e^{i(uS(H)+v\overline{S(H)})}\right) \ (u,v\in\C)\] of the random model
\[S(H)=\frac{X_1+\dots+X_H}{\sqrt{H}},\]
where $H=|I|$. Recall that by \eqref{eq:MpMprobtilde}, we have for all integers $k,r\ge 0$
\begin{equation*}
  \tilde M_q(k,r;I)=\tilde M_\prob(k,r;H)+O\left(c^{3(k+r)}q^{-1/2}H^{\frac{k+r}{2}}\right).
\end{equation*}
Let us fix $0<\varepsilon<1/2$ and let
\begin{equation}
  \label{eq:boundN}
  N=2M\le \varepsilon \frac{\log{q}}{\log(c^6H)}
\end{equation}
be an even integer, so that in particular $c^{6M}q^{-1/2}H^{M}\le q^{-1/2+\varepsilon}$ and
\[\tilde M_q(M,M;I)=\tilde M_\prob(M,M;H)+O(q^{-1/2+\varepsilon}).\]

By Lemma \ref{lemma:charfunctMoments}, we find the following relation between the characteristic functions:
\begin{eqnarray*}
  \phi_{q,I}(u,\overline{u})&=&\phi_H(u,\overline{u})\\
  &&+O\left(\left[\frac{|u|^{N}}{N!}(|\tilde M_\prob(M,M;H)|)+q^{-1/2+\varepsilon}(1+|u|^N)\right]\right).
\end{eqnarray*}

Let $t=M^\alpha/(2\pi)$ for some $\alpha>0$ to be determined later. We apply Corollary \ref{cor:approximationSelberg} after making a change of variable with \eqref{eq:phitildephi} to consider characteristic functions arising from $(u,v)\mapsto u\Re{X}+v\Im{X} \ (u,v\in\R)$ instead of $(u,v)\mapsto u X+v\closure{X} \ (u,v\in\C)$. For all $u,v\in\R$, we then have by H\"older's inequality
  \begin{eqnarray}
    P(\tilde S(t_q,I+x)\in A)&=&P(S(H)\in A)\nonumber\\
           &&+O\left(\frac{1}{t}\int_0^t (|\phi_H(\pi u,\pi u)|+|\phi_H(i\pi u,-i\pi u)|)du\right)\nonumber\\
           &&+O\left(\int_0^t\int_0^t g(u,v)dudv\right)\nonumber\\
           &&+O\left(\frac{1}{t}\int_0^t (g(u,0)+g(0,u))du\right)\label{eq:appproximationSelbergAppl}
\end{eqnarray}
where
\[g(x,y)=(2\pi)^N\left[\frac{x^{N}+y^N}{N!}|\tilde M_\prob(M,M;H)|+q^{-1/2+\varepsilon}(1+x^N+y^N)\right].\]
Let us bound the three error terms in \eqref{eq:appproximationSelbergAppl} one after another:

\begin{enumerate}[leftmargin=*]
\item For the first one, note that
  \[\frac{1}{t}\int_0^t |\phi_H(2\pi u,2\pi u)|du\le \frac{1}{t}\int_\R |\phi_H(2\pi u,2\pi u)|du.\]
  Using Lemma \ref{lemma:charfunctMomentsAverage} and the assumptions on the moments, we have
  \[\phi_H(u,u)=e^{-u^2/2}\left(1+O\left(\frac{|u|^R}{H^{(R-1)/2}}\right)\right)\]
  for $|u|\le H^{\frac{R-1}{2R}}$. Since $\int_\R e^{-u^2/2}<\infty$, the error term becomes $O(1/t)=O\left(M^{-\alpha}\right)$ under the condition
  \begin{equation}
    \label{eq:condtH}
    2\pi t\le H^{\frac{R-1}{2R}}\text{, i.e. }M\le H^{\frac{R-1}{2R\alpha}}.
  \end{equation}
\item The second term $\int_0^t\int_0^t g(u,v)dudv$ is bounded (up to a constant) by
  \begin{eqnarray}
    &&\frac{\tilde M_\prob(M,M;H)}{(2M)!} \frac{(2\pi t)^{2M+2}}{M}\nonumber\\
    &&+q^{-1/2+\varepsilon}\left((2\pi t)^2+\frac{(2\pi t)^{2M+1}}{M}\right).\label{eq:error1}
  \end{eqnarray}
  By Propositions \ref{prop:boundMprobNSD} (non-self-dual case), \ref{prop:boundMprobNSDChar} (Kummer case) and \ref{prop:boundMprobSD} (self-dual case),
  \[\tilde M_\prob(M,M;H)\le
  \begin{cases}
    (M+H-1)^M&\text{non-self-dual case}\\
    M!&\text{Kummer case}\\
    (2M-1)!!&\text{self-dual case}.
  \end{cases}
\]
  By Stirling's approximation, the first summand of \eqref{eq:error1} is bounded (up to a constant) by:
  \begin{itemize}
  \item In the Kummer case: $M^{M(2\alpha-1)+2\alpha-1}$.
  \item In the self-dual case: $M^{2\alpha-\frac{3}{2}+M\left(2\alpha-1+\frac{\log(e/2)}{\log{M}}\right)}$.
  \item In the non-self-dual case:
  \begin{equation}
    \label{eq:boundSD}
    M^{2\alpha-\frac{3}{2}}\left(\frac{e^4}{4}\left(M^{2\alpha-1}+HM^{2\alpha-2}-M^{2\alpha-2}\right)\right)^M\ll M^{2\alpha-\frac{3}{2}}
  \end{equation}
  if $\alpha<1/2$ and under the additional condition $M\gg H^{\frac{1}{2-2\alpha}}$. With \eqref{eq:boundN}, this imposes the more restrictive range
\begin{equation*}
  H=o\left((\log{q})^{\frac{2-2\alpha}{1+\varepsilon(2-2\alpha)}}\right)
\end{equation*}
and the condition
\begin{equation}
  \label{eq:condalpha}
  \frac{1}{2-2\alpha}\le \frac{R-1}{2R\alpha}\text{, i.e. }\alpha\le \frac{R-1}{2R-1}
\end{equation}
because of \eqref{eq:condtH}.
  \end{itemize}

By \eqref{eq:condtH}, the second summand of \eqref{eq:error1} is $O\left(q^{-1/2+2\varepsilon}\right)$ if $\log{H}/\log{q}\le \varepsilon$ since
\begin{eqnarray*}
  (2\pi t)^2&\le& H^{\frac{R-1}{R}}=q^{\frac{\log{H}}{\log{q}}\frac{R-1}{R}}\text{ and }\\
  (2\pi t)^{2M+1}&\le& H^{3M \frac{R-1}{2R}}\le q^{\frac{3(R-1)}{4R}\varepsilon}.
\end{eqnarray*}
\item Under the same conditions, the last error term $\frac{1}{t}\int_0^t (g(u,0)+g(0,u))du$ of \eqref{eq:appproximationSelbergAppl} is bounded by the first one.\\
\end{enumerate}

Hence, the error term in \eqref{eq:appproximationSelbergAppl} is:
\begin{itemize}
\item In the self-dual and Kummer cases
  \[O\left(M^{-\alpha}+M^{2\alpha-1+M\left(2\alpha-1+\frac{\log(e/2)}{\log{M}}\right)}+q^{-\frac{1}{2}+2\varepsilon}\right).\]
  We optimize by taking $\alpha=\frac{M \left(1-\frac{\log(e/2)}{\log{M}}\right)+1}{2M+3}$, which leads to an error term of $O\left(M^{-1/2+\varepsilon}+q^{-1/2+2\varepsilon}\right)$.
\item In the non-self-dual case, $O\left(M^{-\alpha}+q^{-1/2+2\varepsilon}\right)$ (since $2\alpha-3/2\le -\alpha$ when $\alpha\le 1/2$). By \eqref{eq:condalpha}, we optimize by taking $\alpha=\frac{R-1}{2R-1}$ and we obtain the error term $O\left(M^{-\frac{R-1}{2R-1}}+q^{-1/2+2\varepsilon}\right)$ for the range
  \[H=o\left((\log{q})^{\frac{2R}{(2R-1)(1+2\varepsilon)}}\right).\]
\end{itemize}

Finally, after letting
\[M=\ceil{\min\left(H^{\frac{R-1}{2R\alpha}},\frac{\varepsilon}{2}\frac{\log{q}}{\log(c^6H)}\right)}\to+\infty,\]
we can apply Proposition \ref{prop:BerryEsseen} to $S(H)$, and combining with \eqref{eq:appproximationSelbergAppl} gives Theorem \ref{thm:gaussian2}.

\section{Traces of random matrices in classical groups}

In this section, we prove Proposition \ref{prop:momentsRM}, which will conclude the proof of Theorem \ref{thm:gaussian2}. In comparison to earlier works, recall that it is important for us to obtain bounds on moments of high order with respect to the rank.

\subsection{Special linear case}\label{subsec:SLcase}

\begin{proposition}\label{prop:momentsSL}
  Let $N\ge 2$ and let $X=\tr\theta$, where $\theta$ is a random variable uniformly distributed in $\SU_{N}(\C)$ with respect to the Haar measure. For $k,r\ge 0$ integers, let us consider the moment $M(k,r)=\E(X^k\overline{X}^r)$. Then:
\begin{enumerate}
\item\label{item:momentsSL1} We have
  \[M(k,r)=\delta_{N\mid k-r}\sum_{\substack{\lambda\vdash k\\ l(\lambda)\le N, \, \lambda_N\ge -a}} \dim S_{\lambda}\dim S_{\lambda+(a^N)}\]
where $a=(k-r)/N$ and $S_\lambda$, respectively $S_{\lambda+a}$, is the Specht $\Sf_k$-module (resp. $\Sf_r$-module) associated to the partition $\lambda$, respectively $\lambda+(a^N)=\lambda+(a,\dots,a)$.
\item\label{item:momentsSL3} $M(k,r)\le \sqrt{k!r!}$.
\item\label{item:momentsSL4} $M(k,k)=k!$ if $k\le N$.
\item\label{item:momentsSL5} $M(k,r)=0$ if $k,r<N$ and $k\neq r$.
\end{enumerate}
\end{proposition}
\begin{proof}
  We use the same technique as in \cite{DiaconisShahshahani}, but we also need to handle the case $k,r\ge N$.

  Let $\Std$ be the standard representation of $\SU_N(\C)$ in $\GL_{N}(\C)$. Recall that the irreducible representations of $\SL_N(\C)$ (and hence of its maximal compact subgroup $\SU_ N(\C)$) are the Schur-Weyl modules $S_\lambda(\Std)$ indexed by partitions $\lambda$ of length $l(\lambda)\le N$ (see \cite[15.3]{FulHar91}). Moreover, the character of $S_\lambda(\Std)$ is given by the Schur polynomial $s_\lambda$ evaluated on the eigenvalues (see \cite[I.3]{Macdonald} or \cite[6.1]{FulHar91}). For $\lambda=(\lambda_1,\dots,\lambda_l)$, recall the \textit{power symmetric polynomials}
  \[p_\lambda=p_{\lambda_1}\dots p_{\lambda_l}\text{ where }p_m=x_1^m+\dots+x_N^m\text{ for any }m\in\N.\]

By the representation theory of the symmetric group and the theory of symmetric polynomials (see \cite[I.7.8]{Macdonald}), we have the decomposition of $p_\lambda$ into the basis of Schur polynomials: for any partition $\lambda$ of length $\le k$,
  \[p_\lambda=\sum_{\mu\vdash k}\chi_\mu(\lambda)s_\mu,\]
  where $\chi_\mu(\lambda)$ is the character of the irreducible Specht $\Sf_k$-module $S_\mu$ corresponding to $\lambda$, evaluated on the conjugacy class corresponding to $\lambda$. In particular,
  \[(x_1+\dots+x_N)^k=\sum_{\substack{\mu\vdash k\\ l(\mu)\le N}}\dim S_\mu s_\mu(x_1,\dots,x_N).\]
   Since $(x_1+\dots+x_N)^k$ (resp. $s_\mu(x_1,\dots,x_N)$) is the character of $\Std^{\otimes k}$ (resp. of the irreducible representation $S_\mu(\Std)$) evaluated at a matrix whose eigenvalues are $x_1,\dots,x_N$, we get by orthogonality that $M(k,r)$ is equal to
   \begin{equation}
     \label{eq:MkrSL}
     \int_{\SU_N(\C)} \tr(g)^k\overline{\tr(g)^r}dg=\sum_{\substack{\mu_1\vdash k\\ l(\mu_1)\le N}}\sum_{\substack{\mu_2\vdash r\\ l(\mu_2)\le N}} \dim S_{\mu_1}\dim S_{\mu_2}\delta_{S_{\mu_1}(\Std)\cong S_{\mu_2}(\Std)}.
   \end{equation}
The Cauchy-Schwarz inequality yields
\begin{equation}
  \label{eq:MkrSLCS}
  M(k,r)^2\le\sum_{\substack{\mu_1\vdash k\\l(\mu_1)\le N}}(\dim S_{\mu_1})^2\sum_{\substack{\mu_2\vdash r\\l(\mu_2)\le N}}(\dim S_{\mu_2})^2\le k!r!
\end{equation}
since the Specht modules $S_\mu$ ($\mu\vdash k$) give the irreducible representations of the symmetric group $\Sf_k$ (see \cite[I.7]{Macdonald}). Hence we obtain \ref{item:momentsSL3}.

Next, note that $S_{\mu_1}(\Std)\cong S_{\mu_2}(\Std)$ if and only if $\mu_2=\mu_1+(a^{N})$ for some $a\in\Z$ (see \cite[p. 223]{FulHar91}). If the latter holds, we have $N\mid k-r$, $a=(k-r)/N$, $l(\mu_1)\le N$ and $(\mu_1)_N\ge -a$. Thus \eqref{eq:MkrSL} becomes
\[M(k,r)=\sum_{\substack{\lambda\vdash k\\ l(\lambda)\le N, \, \lambda_N\ge -a}} \dim S_{\lambda}\dim S_{\lambda+(a^N)}\]
if $N\mid k-r$ and $0$ otherwise. This gives \ref{item:momentsSL1}.

Let us now assume that $k\le N$. We then automatically have $l(\lambda)\le k\le N$ for every partition $\lambda$ of $k$. If moreover $k=r$, then $a=0$ and
\[M(k,k)=\sum_{\lambda\vdash k}(\dim S_\lambda)^2=k!,\]
which is \ref{item:momentsSL4}. Finally, if $0\le k,r<N$ are distinct, then $N\nmid k-r$ and $M(k,r)=0$, which is \ref{item:momentsSL5}.
\end{proof}
\begin{remark}(see also Remarks \ref{rem:rangeNSD} and \ref{rem:rangeNSD2}). The second bound we have given in \eqref{eq:MkrSLCS} is not asymptotically tight for $k\neq r$. However, replacing it by a better asymptotic would not improve the results (or in particular recover the range $\log{H}=o(\log{q})$ in the non-self-dual case of Theorem \ref{thm:gaussian2}). Indeed, Regev \cite[Corollary 4.4]{RegevYoungDiagrams} used the hook-length formula to show that as $k\to\infty$, we have
\[\sum_{\substack{\lambda\vdash k\\ l(\lambda)\le N}} (\dim S_\lambda)^2\sim C(N)\frac{N^{2k}}{k^{(N^2-1)/2}},\]
where $C(N)=N^{N^2/2}\left(2\pi\right)^{(1-N)/2}2^{(1-N^2)/2}\prod_{n=1}^{N-1} n!$. The bound \eqref{eq:boundSD} becomes
\[\left((1+\varepsilon)C(R+1)\right)^H M^{2\alpha-\frac{3}{2}} \left(HM^{2\alpha-2}(R+1)^2\right)^M\]
for any $\varepsilon>0$, for which we still need the restricted range $M>H^{\frac{1}{2-2\alpha}}$.
\end{remark}
\subsection{Symplectic case}

\begin{proposition}\label{prop:momentsSp}
  Let $N\ge 1$ and $X=\tr\theta$, where $\theta$ is a random variable uniformly distributed in $\USp_{2N}(\C)=\Sp_{2N}(\C)\cap U_{2N}(\C)$ with respect to the Haar measure. For $k\ge 0$ an integer, let us consider the moment $M(k)=\E(X^k)$. Then
  \begin{enumerate}
  \item\label{item:momentsSp1} $M(k)=0$ if $k$ is odd.
  \item\label{item:momentsSp2} $M(k)\le (k-1)!!$ if $k$ is even, with equality if $k\le N$.
  \end{enumerate}
\end{proposition}
\begin{proof}
  Let $\Std$ be the standard representation of $\Sp_{2N}(\C)$. As in the simple linear case, recall that the irreducible representations of $\Sp_{2N}(\C)$ (and hence of $\USp_{2N}(\C)$) are given by the Weyl modules $S_{\langle\mu\rangle}(\Std)$ indexed by partitions $\mu$ with $l(\mu)\le N$ (\cite[17.3]{FulHar91}). By Peter-Weyl, $M(k)=\mult_1(\Std^{\otimes k})$. By \cite[Theorem 6.15]{Sundaram}, we have the decomposition
  \[\Std^{\otimes k}=\bigoplus_{\substack{\mu\\ l(\mu)\le N}} f_\mu^k(N) S_{\langle\mu\rangle}(\Std),\]
  where $f_\mu^k(N)$ is the number of sequences of partitions $(\varnothing=\mu_0,\dots,\mu_k=\mu)$ such that
  \begin{enumerate}[(a)]
  \item\label{item:differOneBoxSp} two consecutive partitions differ by exactly one box in their Young diagrams, and
  \item $l(\mu_i)\le N$ for all $i$.
  \end{enumerate}
  Hence, $M(k)=f_{0}^k(N)$, so that \ref{item:momentsSp1} is clear. By \cite[Lemma 8.3]{Sundaram}, when $k$ is even, the number $f_\mu^{k}$ of sequences of partitions $(\varnothing=\mu_0,\dots,\mu_k=\mu)$ verifying \ref{item:differOneBoxSp} satisfies $f_0^{k}=(k-1)!!$, whence \ref{item:momentsSp2} since $f_0^{2k}(N)\le f_0^{2k}$, with equality if $k\le N$ since then $l(\mu_i)\le i\le k$.
\end{proof}
\begin{remark}\label{rem:momentsSpDiaconis}
  When $k\le N$, this is proven in \cite[Theorem 6]{DiaconisShahshahani} by using the analogue for $\Sp$ of the Schur-Weyl duality, through the Brauer algebra $D_f(-2N)$, following results of Wenzl and Ram (see in particular \cite[Theorem 4.4 (c), Corollary 4.5 (c)]{RamCharBrauer}). However, this cannot be exploited when $k>N$ since $D_f(-2N)$ is not semisimple in that case.
\end{remark}

\subsection{Special orthogonal case}

\begin{proposition}\label{prop:momentsSO}
  Let $N\ge 2$ and $X=\tr\theta$, where $\theta$ is a random variable uniformly distributed in $\SO_{N}(\R)$ with respect to the Haar measure. Let us consider the moment $M(k)=\E(X^k)$ for $k\ge 0$ an integer. Then:
  \begin{enumerate}
  \item $M(k)=0$ if $k$ is odd.
  \item $M(k)\le (k-1)!!$ if $k$ is even, with equality if $k\le \floor{N/2}$.
  \end{enumerate}
\end{proposition}
\begin{proof}
  This is similar to the symplectic case. Let $\Std$ be the standard representation of $\SO_{N}(\R)$.
  \begin{enumerate}
  \item (Case $N=2N'+1$ odd). By \cite[Theorem 4.2]{Sundaram2}, we have the decomposition
 \[\Std^{\otimes k}=\bigoplus_{\substack{\mu\\ l(\mu)\le N'}} F_\mu^k(N') S_{[\mu]}(\Std),\]
 where $S_{[\mu]}(\Std)$ is the irreducible representation of $\SO_{2N'+1}(\R)$ associated to the partition $\mu$ (obtained from the Weyl module, see \cite[19.5]{FulHar91}) and $F_\mu^k(N')$ is the number of sequences of partitions $(\varnothing=\mu_0,\dots,\mu_k=\mu)$ such that
 \begin{enumerate}[(a)]
 \item two consecutive partitions either differ by exactly one box in their Young diagrams, or are equal of length $N'$, and
 \item $l(\mu_i)\le N'$ for all $i$.
 \end{enumerate}
Hence, $M(k)=F_0^k(N')$. Clearly, $F_0^k(N')\le f_0^k(N')\le f_0^{k}$ with equality if $k\le N'$, where $f_0^k(N')$ and $f_0^k$ are as in the proof of Proposition \ref{prop:momentsSp}. The result follows then from the latter.
  \item (Case $N=2N'$ even). By \cite[Corollary 4]{Pro90}, we have for $\SO_{2N'}(\R)$ the decomposition
    \[\Std^{\otimes k}=\bigoplus_{\substack{\mu\\ l(\mu)\le N'}} G_\mu^k(N') S_{[\mu]}(\Std),\]
    where $G_\mu^k(N')$ is the number of sequences of partitions $(\varnothing=\mu_0,\dots,\mu_k=\mu)$ such that:
    \begin{enumerate}
    \item two consecutive partitions differ by exactly one box in their Young diagrams, and
    \item for every $0\le i\le k$, the sum of the length of the first two columns in the Young diagram of $\mu_i$ is $\le N'$.
    \end{enumerate}
    Thus, we have again $G_\mu^k(N')\le f_0^k(N')\le f_0^k$ with equality if $k\le N'$, since the Young diagram of $\mu_i$ contains at most $i\le k$ boxes.
  \end{enumerate}
\end{proof}
\begin{remark}
  As for the symplectic case (see Remark \ref{rem:momentsSpDiaconis}), this is proved when $k\le N$ in \cite[Theorem 4]{DiaconisShahshahani}, by using \cite[Theorem 4.4 (b), Corollary 4.5 (b)]{RamCharBrauer}, but again this method cannot be applied when $k>N$.

  The idea of Sundaram in \cite{Sundaram} and \cite{Sundaram2} is to define tableaux generalizing the Robinson-Schensted-Knuth correspondence and to prove a generalized insertion scheme. The symplectic case actually goes back to Berele, and the odd-dimensional orthogonal case is an extension of the latter. For orthogonal groups, there are also generalized tableaux by King-Welsh, Koike-Terada and Fulmek-Krattenhalter, but these do not have at first an easy combinatorial description.
\end{remark}

\section{Examples: coherent families}\label{sec:examples}

In this final section, we give examples of coherent families arising from the examples of Section \ref{subsubsec:examplesTF}, so that Theorems \ref{thm:gaussian1} and \ref{thm:gaussian2} apply.

The construction of the sheaves and the computation of their monodromy groups come from Katz's works \cite{KatzGKM,KatzSarnak91}. It usually remains to argue that the conductor is bounded independently from $q$, show the independence of shifts and to show that the arithmetic and geometric monodromy groups coincide (eventually up to twisting, see Section \ref{subsec:examplesTwisting}). We start by two technical sections with tools to do so, before treating the examples successively.

\subsection{Independence of shifts}
Showing that a geometric isomorphism of the form \eqref{eq:translatinIsom} does not exist can usually be done by looking at the ramification on both sides.

\begin{lemma}\label{lemma:translationInvariance}
  Let $\Fc$ be a nontrivial $\ell$-adic sheaf over $\F_q$ and let $a\in\G_m(\F_q)$ such that there exists a geometric isomorphism of the form \eqref{eq:translatinIsom}. Then
  \begin{enumerate}
  \item \label{item:translationInvariance:SingL} $\Sing(\Fc)\Delta(\Sing(\Fc)-a)\subset\Sing(\Lc)\subset\Sing(\Fc)\cup(\Sing(\Fc)-a)$, where $\Delta$ denotes the symmetric difference.
  \item \label{item:translationInvariance:Invariants} If $\Sing(\Fc)\cap\A^1(\F_q)\neq\varnothing, \ \A^1(\F_q)$, there exists $x\in\Sing(\Fc)\cap\A^1(\F_q)$ such that $\Fc^{I_x}=0$.
  \item \label{item:translationInvariance:LGeomTrivial} If $\Sing(\Fc)\neq\varnothing,\{\infty\}$, the sheaf $\Lc$ is not geometrically trivial.
  \item \label{item:translationInvariance:EP} If $\Lc$ is not geometrically trivial,
    \begin{equation}
      \label{eq:EP}
      |\Sing(\Lc)|+\sum_{x\in\Sing(\Lc)} \Swan_x(\Lc)\ge 2.
    \end{equation}
  \item If $\Fc$ has unique break $t\in\R_{\ge 0}$ at $x\in\P^1(\F_q)$, then the break decomposition of $\Fc\otimes\Lc$ at $x$ is
    \begin{equation}
    \label{eq:breakDecomposition}
    \Fc\otimes\Lc=
  \begin{cases}
    (\Fc\otimes\Lc)(\Swan_\infty(\Lc))&: t<\Swan_\infty(\Lc)\\
    (\Fc\otimes\Lc)(t)&: t>\Swan_\infty(\Lc)\\
    \sum_{z\le t} (\Fc\otimes\Lc)(z)&: t=\Swan_\infty(\Lc).
  \end{cases}
  \end{equation}
\item \label{item:translationInvariance:UniqueBreak} If $\Fc$ has unique break $t\in\R_{\ge 0}$ at $\infty$, then $\Swan_\infty(\Lc)\le t$. If $t$ is not an integer, then $\Swan_\infty(\Lc)\le\floor{t}$.
  \end{enumerate}
\end{lemma}
\begin{proof}
  \begin{enumerate}
  \item This is clear.
  \item If $x\in\Sing(\Lc)-\Sing(\Fc)$, then
    \[\Fc^{I_{x+a}}\cong([+a]^*\Fc)^{I_x}\cong(\Fc\otimes\Lc)^{I_x}=\Fc\otimes\Lc^{I_x}=0.\]
    In particular, by \ref{item:translationInvariance:SingL}, if $y\in\Sing(\Fc)$ but $y-a\not\in\Sing(\Fc)$, then $\Fc^{I_y}=0$. If $x\in\Sing(\Fc)\cap\A^1(\F_q)$ and $\A^1(\F_q)\not\subset\Sing(\Fc)$, there exists an integer $m\ge 1$ such that $y=x-(m-1)a\in\Sing(\Fc)$ but $x-ma\not\in\Sing(\Fc)$, whence the conclusion.
  \item By \ref{item:translationInvariance:SingL}, $\Lc$ is not lisse under the assumptions.
  \item The Euler-Poincaré formula of Grothendieck-Ogg-Safarevich gives that the left-hand side of \eqref{eq:EP} is equal to
    \[2+\dim H^1_c(U_\Lc\times\overline\F_q,\Lc)\ge 2\]
    if $\Lc$ is nontrivial.
  \item This follows from \cite[Lemma 1.3]{KatzGKM}.
  \item By \eqref{eq:breakDecomposition}, we have
 \[\Swan_\infty(\Fc\otimes\Lc)=\begin{cases}
    \rank(\Fc)\Swan_\infty(\Lc)&: t<\Swan_\infty(\Lc)\\
    \rank(\Fc)t&: t>\Swan_\infty(\Lc).
  \end{cases}\]
  On the other hand, by \eqref{eq:translatinIsom},
  \[\Swan_\infty(\Fc\otimes\Lc)=\Swan_\infty([+a]^*\Fc)=\Swan_\infty(\Fc)=t\rank(\Fc),\]
  which implies that the case $t<\Swan_\infty(\Lc)$ cannot hold. The last statement follows from the fact that the Swan conductor is an integer.
  \end{enumerate}
\end{proof}

The following classification result will also be useful:
\begin{lemma}\label{lemma:FKMClass}
  Let $\Fc$ be a geometrically irreducible $\ell$-adic sheaf over $\F_q$.
  \begin{enumerate}
  \item \label{item:FKMClass1} If $\Sing(\Fc)=\varnothing$, then $\Fc$ is geometrically trivial.
  \item \label{item:FKMClass2} If $|\Sing(\Fc)|=1$ and $\Fc$ is tamely ramified, then $\Fc$ is geometrically trivial.
  \item \label{item:FKMClass3} If $\Sing(\Fc)=\{x,y\}$ for $x,y\in\P^1(\F_q)$ distinct and $\Fc$ is tamely ramified, then there exists a multiplicative character $\chi: \F_q^\times\to\C^\times$ and a geometric isomorphism
    \[\Fc\cong\Lc_{\chi((X-y)/(X-y))}.\]
  \item \label{item:FKMClass4} If $\Sing(\Fc)=\{x\}$ and $\Swan_x(\Fc)\le 1$, there exists an additive character $\psi: \F_q\to\C^\times$ and a geometric isomorphism
    \[\Fc\cong
      \begin{cases}
        \Lc_{\psi}&:x=\infty\\
        \Lc_{\psi(1/(X-x))}&:x\neq\infty.
      \end{cases}
\]
  \end{enumerate}
\end{lemma}
\begin{proof}
  See \cite[Proposition 4.4.6]{FKMCours}.
\end{proof}

\subsubsection{Arguments with unipotent blocks}

\begin{lemma}\label{lemma:translIsomUnipPrelim}
  Let $\Gc$ an $\ell$-adic sheaf over $\F_q$ such that $\Sing(\Gc)\cap\A^1(\F_q)\neq \varnothing, \ \A^1(\F_q)$. For every $s\in\Sing(\Gc)\cap\A^1(\F_q)$, we consider the tame part of the break decomposition of $\Gc$ at $s$,
  \begin{equation}
    \label{eq:breakDecGsTame}
    \Gc(s)^{\textnormal{tame}}=\bigoplus_\chi \left(\operatorname{Unip.}\otimes\Lc_{\chi(X+s)}\right),
  \end{equation}
  and we assume that either the trivial multiplicative character $\chi=1$ appears, or that at least two distinct characters $\chi_1,\chi_2$ appear. Then there is no isomorphism of the form \eqref{eq:translatinIsom} with $a\neq 0$.
\end{lemma}
\begin{proof}
  Let us assume that there is an isomorphism of the form \eqref{eq:translatinIsom} for $\Gc$ with $a\neq 0$.
  If the break decomposition of $\Gc$ at some $s\in\Sing(\Gc)\cap\A^1(\F_q)$ does not contain a summand $\operatorname{Unip.}\otimes\Lc_{\chi(X+s)}$ with $\chi$ trivial, we replace $\Gc$ by $\Gc\otimes\Lc_{\overline\chi_1(X+s)}$, where $\chi_1$ is a character appearing in \eqref{eq:breakDecGsTame}. This new sheaf still satisfies the same hypotheses as $\Gc$, with the same $a$ in \eqref{eq:translatinIsom} (but with a different $\Lc$), and with a unipotent block in the break decomposition at $s$.

Recursively, we can hence assume that the tame part of $\Gc$ at any $s\in\Sing(\Gc)\cap\A^1(\F_q)$ contains a unipotent block.

By Lemma \ref{lemma:translationInvariance} \ref{item:translationInvariance:Invariants}, there exists $s\in\Sing(\Gc)\cap\A^1(\F_q)$ such that $\Gc^{I_s}=0$, a contradiction.
\end{proof}
\begin{lemma}\label{lemma:translIsomUnip}
  Let $\psi: \F_q\to\C^\times$ be a nontrivial additive character and let $\Gc=\FT_\psi(\Fc)$ be the $\ell$-adic Fourier transform of a Fourier sheaf $\Fc$ over $\F_q$, with $\rank(\Fc)<q-1$. For all $s\in\A^1(\F_q)$, we consider the break decomposition of $\Fc^{(s)}=\Fc\otimes\Lc_{\psi(sX)}$ at $\infty$:
  \begin{eqnarray}
    \Fc^{(s)}&=&\bigoplus_{t\in\R_{\ge 0}}\Fc^{(s)}(t)=\Fc^{(s),\textnormal{tame}}\bigoplus\Fc^{(s),\textnormal{wild}}\nonumber\\
    &=&\left(\bigoplus_\chi \left(\operatorname{Unip}(\chi,s)\otimes\Lc_{\chi(X+s)}\right)\right)\bigoplus \left(\bigoplus_{t>0}\Fc^{(s)}(t)\right).\label{eq:breakDecFcs}
  \end{eqnarray}
  We assume that:
\begin{itemize}
\item The decomposition \eqref{eq:breakDecFcs} at $s=0$ contains at least one break $t\in[0,1]$.
\item For all $s\in\A^1(\F_q)$ such that the decomposition \eqref{eq:breakDecFcs} contains a break $t\in[0,1)$, either the trivial multiplicative character appears in the tame part, or the latter contains at least two distinct characters.
\end{itemize}
Then there is no isomorphism of the form \eqref{eq:translatinIsom} for $\Gc$ with $a\neq 0$.
\end{lemma}
\begin{proof}
  By \cite[Corollary 8.5.8]{KatzGKM} (see also \cite[Corollary 7.4.5]{KatzESDE}), the first assumption and the condition on the rank of $\Fc$ imply that $\Sing(\Gc)\cap\A^1(\F_q)\neq \varnothing, \ \A^1(\F_q)$. Moreover, $s\in\Sing(\Gc)\cap\A^1(\F_q)$ if and only if the decomposition \eqref{eq:breakDecFcs} contains a break $t\in[0,1)$. By \cite[7.4.4(3)]{KatzESDE}, the tame part of the break decomposition of $\Gc$ at $s$ is in this case
\[\bigoplus_\chi \left(\operatorname{Unip}(\chi,s)\otimes\Lc_{\overline\chi(X+s)}\right)\]
(with the same unipotent blocks). It suffices to apply Lemma \ref{lemma:translIsomUnipPrelim} to conclude.
\end{proof}
\subsection{Equality of arithmetic and geometric monodromy groups}\label{subsec:arithmGeomMono}
  In \cite{KatzESDE}, often only the geometric monodromy group $G_\geom=G_\geom(\Fc)$ of an $\ell$-adic sheaf $\Fc$ over $\F_q$, or its connected component
\[G^0_\geom\le G_\geom\le G_\arith=G_\arith(\Fc),\]
are directly given. As is explained in \cite[7.11--7.14]{KatzESDE} and \cite{MichelMinorationsSommesExp}, it is usually possible to get
\[G_\geom^0=G_\geom=G_\arith,\]
up to twisting $\Fc$ by a rank $1$ sheaf, or even, ideally, a constant:

\begin{itemize}[itemsep=0.2cm]
\item (Symplectic case) This is the simplest case. Proving that $G_\geom^{0}=\Sp_n(\C)$ with the techniques in \cite[Chapter 7]{KatzESDE} actually shows that the sheaf is itself symplectically self-dual (see \cite[7.13, p. 244]{KatzESDE}), as for Kloosterman sheaves (see \cite[4.1.11]{KatzGKM}). Hence $G_\arith\subset\Sp_n(\C)$ and thus $G_\geom=G_\arith=\Sp_n(\C)$. 
\item (Special orthogonal case) Similarly, proving that $G_\geom^0=\SO_n(\C)$ (or $\O_n(\C)$) with the techniques of \cite[Chapter 7]{KatzESDE} actually shows that $G_\arith\subset \O_n(\C)$ (see \cite[7.14, $\O$-Example(2)]{KatzESDE}). Hence, there exists $\alpha\in\{\pm 1\}$ such that $\Fc'=\alpha^{1/n}\otimes \Fc$ has $G_\geom(\Fc')=G_\arith(\Fc')=\SO_n(\C)$.
\item (Special linear case) This is the hardest case. Assume that $G_\geom^0=G_\geom^{0,\der}=\SL_n(\C)$. We can determine the geometric determinant $\det(\Fc)$ and twist it by a rank one sheaf $\Lc$ to make it geometrically trivial, hence arithmetically isomorphic to $\alpha\otimes\overline\Q_\ell$, for a Weil number $\alpha$ of weight $0$ (which may be difficult to determine explicitly). If we let $\Fc'=\alpha^{-1/n}\otimes \Lc\otimes\Fc$, we have $G_\arith(\Fc')\subset\SL_n(\C)$ and $\SL_n(\C)=G^{0}_\geom\subset G^{0}_\geom(\Fc')$ since $G_\geom^{0}$ is equal to its derived subgroup and $\Lc$ has rank one. This gives
\[G_\geom(\Fc')=G_\arith(\Fc')=\SL_n(\C).\]
Moreover, it happens in some cases that $\Lc$ is arithmetically constant, so that $\Fc'=\alpha^{-1/n}\otimes\Fc$ is simply a renormalization of $\Fc$.
\end{itemize}

\subsection{Kummer sheaves: multiplicative characters}
\begin{proposition}
  A family $(\Fc)_q$ of Kummer sheaves $\Lc_{\chi(f)}$, where $\deg(f)$ is bounded independently from $q$ and $f$ has no zero or pole of order divisible by $\ord(\chi)$, is coherent.
\end{proposition}
\begin{proof}
  For the construction of the Kummer sheaf, see \cite[Exposé 6, Section 1]{DelEC} or \cite[Section 4.3]{KatzGKM}. We have $\cond(\Lc_{\chi(f)})=1+\deg(f_1)+\deg(f_2)$ where $f=f_1/f_2$ with $f_1,f_2\in\F_q[X]$, $(f_1,f_2)=1$.
\end{proof}
\subsection{Kloosterman sheaves}

\begin{theorem}[Deligne, Katz]
  For $n\ge 2$ an integer, there exists a \emph{Kloosterman $\ell$-adic sheaf} $\Klc_n$ over $\F_q$, of rank $n$, with trace function equal to the Kloosterman sum \eqref{eq:KS}. The family $(\Klc_{n,q})_{q\text{ odd}}$ is coherent, with monodromy group equal to
\[
\begin{cases}
  \SL_n(\C)&: n\text{ odd}\\
  \Sp_n(\C)&: n\text{ even}.
\end{cases}
\]
\end{theorem}
\begin{proof}
  For the construction and computation of monodromy groups, see \cite{KatzGKM}. We have $\cond(\Klc_n)=n+3$. Finally, the independence of shifts follows from Lemma \ref{lemma:translIsomUnip}, which can be applied thanks to \cite[7.4.1]{KatzGKM}.
\end{proof}

\subsection{Hypergeometric sheaves}

\begin{proposition}[Katz]\label{prop:hypergeometric}
  Let $n\ge m\ge 0$ be integers with $r=m+n\ge 1$, $\bs\chi_q=(\chi_{i,q})_{1\le i\le n}$, $\bs\rho_q=(\rho_{j,q})_{1\le j\le m}$ tuples of pairwise distinct characters of $\F_q^\times$. There exists a \emph{hypergeometric sheaf} $\Hc(\bs\chi_q,\bs\rho_q)$ over $\F_q$ of rank $n$, with trace function equal to the \textit{hypergeometric sum} $\textnormal{Hyp}(\bs\chi_q,\bs\rho_q): \F_q\to\C$ defined by
\[t\mapsto \frac{(-1)^{r-1}}{q^{(r-1)/2}}\sum_{\substack{\bs x\in\F_q^n, \bs y\in\F_q^m\\N(\bs x)=tN(\bs y)} } \left(\prod_{i=1}^n \chi_{i,q}(x_i) \prod_{j=1}^m \overline{\rho_{j,q}(y_j)}\right) e\left(\frac{\tr(T(\bs x)-T(\bs y))}{p}\right),\]
where $N: \F_q^n\to\F_q$ is the norm (product of components) and $T: \F_q^n\to\F_q$ the trace (sum of components). We assume that $\Lambda_q=\prod_i \chi_{i,q}=1$ and either:
  \begin{enumerate}
  \item $n=m$ is odd and $\Gamma_q=\prod_j \rho_{j,q}=1$ is constant, or
  \item $n-m\ge 3$ is odd.
  \end{enumerate}
  Then the family $(\Hc(\bs\chi_q,\bs\rho_q))_q$ is coherent with monodromy group $\SL_n(\C)$.
\end{proposition}
\begin{proof}
  The construction can be found in \cite[Theorem 8.4.2]{KatzESDE}. We find that $\cond(\Hc(\bs\chi,\bs\rho))=n+3$.

  The connected component at the identity $G^0_\geom$ is computed in \cite[Theorems 8.11.2, 8.11.3]{KatzESDE}, and can be $\SL_n(\C), \ \Sp_n(\C), \ \SO_n(\C)$, plus some exceptional cases in low rank. Moreover, $G_\geom^0=G^{0,\der}_\geom$. The distinction between the possible cases is not straightforward (see \cite[p. 291]{KatzESDE}), but $G_\geom^0=G^{0,\der}_\geom=\SL_n(\C)$ if either
  \begin{enumerate}
  \item $n=m$ is odd, $\Lambda_q=1$ or
  \item $n-m\ge 3$ is odd.
  \end{enumerate}
To make the arithmetic and geometric monodromy group coincide, we use the strategy of Section \ref{subsec:arithmGeomMono}. By the computation of the arithmetic determinant in \cite[8.12]{KatzESDE}, there is an explicit Weil number $\alpha=\alpha(\bs\chi,\bs\rho)\in\overline{\Q}_\ell$ of weight 0 such that $\det\Hc(\bs\chi,\bs\rho)\cong \alpha\otimes \Lc$ with
\[\Lc=
\begin{cases}
  \Lc_\Lambda\otimes[x\mapsto 1-x]^*\Lc_{\Gamma/\Lambda}&\text{ if }n=m,\\
  \Lc_\psi\otimes\Lc_\Lambda&\text{ if }n-m=1,\\
  \Lc_\Lambda&\text{ if }n-m\ge 2.
\end{cases}
\]
Under the assumptions of the proposition, $\Lc$ is arithmetically trivial and $\alpha=1$.

The break decomposition of the hypergeometric sheaf is determined recursively in \cite[Theorem 8.4.2(6)]{KatzESDE}, and the independence of shifts is then a consequence of Lemma \ref{lemma:translIsomUnip}.
\end{proof}

\begin{example}
  Thus, families of hypergeometric sums of the form
\[\frac{(-1)^{r-1}}{q^{(r-1)/2}}\sum_{\substack{\bs x\in\F_q^n, \bs y\in\F_q^n\\N(\bs x)=tN(\bs y)} } \left(\prod_{i=1}^{n-1} \chi_i(x_ix_n^{-1})\overline{\rho_i(y_iy_n^{-1})}\right)  e\left(\frac{\tr(T(\bs x)-T(\bs y))}{p}\right) \ (t\in\F_q)\]
with $n$ odd or
\[\frac{(-1)^{r-1}}{q^{(r-1)/2}}\sum_{\substack{\bs x\in\F_q^n, \bs y\in\F_q^m\\N(\bs x)=tN(\bs y)} } \left(\prod_{i=1}^n \chi_i(x_ix_n^{-1})\prod_{j=1}^m \overline{\rho_j(y_j)}\right) e\left(\frac{\tr(T(\bs x)-T(\bs y))}{p}\right) \ (t\in\F_q)\]
with $n-m\ge 3$ odd, are coherent.

  For $m=0$ and $\bs\chi=(1)_{1\le i\le n}$, we recover the Kloosterman sheaf $\Klc_{n}$.
\end{example}

\subsection{General exponential sums of the form \eqref{eq:expsumGen}}
\begin{proposition}\label{prop:genExpSum}
  Let $f,g,h\in\Q(X)$. If $q$ is large enough to consider $f,g,h\in\F_q(X)$ and $g$ (resp. $h$) has no pole or zero (resp. no pole) of order divisible by $p$, we consider the sheaves
  \[\Fc_1=\Lc_{\psi(h)}\otimes\Lc_{\chi(g)}, \ \Fc_2=f_*\Fc_1,\]
  for $\psi:\F_q\to\C$ (resp. $\chi:\F_q^\times\to\C$) an additive (resp. multiplicative) character. If $\Fc_2$ is a Fourier sheaf, then there exists a sheaf $\Gc=\FT_\psi(\Fc_2)$ (the $\ell$-adic Fourier transform of $\Fc_2$) with trace function given by \eqref{eq:expsumGen}. Moreover, $\cond(\Gc)$ is bounded above independently from $q$.
\end{proposition}
\begin{proof}
  The construction of the $\ell$-adic Fourier transform can be found in \cite[Chapters 5, 8]{KatzGKM}. The uniform bound on the conductor follows from the general bound on conductors of Fourier transforms \cite[Proposition 8.2]{AlgebraicTwists}, obtained from Laumon's analysis of the ramification of $\ell$-adic Fourier transforms \cite[7.3--7.5]{KatzESDE}.
\end{proof}

We can distinguish the following cases:
\begin{enumerate}[(i)]
\item $h=0$ and $\chi=1$, so that $\Fc_1$ is the trivial sheaf. These are sums of the form \eqref{eq:FouvryMichel}, studied in \cite[7.10]{KatzESDE} and by Fouvry-Michel in \cite{MichelMinorationsSommesExp}, \cite{FouvryMichelRecherche} and \cite{FouvryMichelSommes}.
\item $\Fc_1$ is nontrivial and $f=X$. More particularly, we consider the case $\chi=1$ and $h$ a polynomial of degree $n\ge 2$, which includes Birch sums \eqref{eq:BS}. These are studied in \cite[7.12]{KatzESDE} and \cite{KatzMonodromyFamES}.
\item $\Fc_1$ is nontrivial and $f\neq X$. More particularly, we will consider the case studied in \cite[7.7, 7.13, 7.14]{KatzESDE} where $h$ is odd with a pole of order $\ge 1$ at $\infty$, $f\neq 0$ is an odd polynomial, and there exists an even or odd rational function $L$ with $g(x)g(-x)=L(x)^{\ord(\chi)}$.
\end{enumerate}

\subsubsection{Independence of shifts}

The following criterion generalizes the argument of \cite{FKMSumProducts} for Birch sums to sheaves of the general form of Proposition \ref{prop:genExpSum} and allows to reduce to the case of $\Lc$ being an Artin-Schreier sheaf in a geometric isomorphism of the form \eqref{eq:translatinIsom}.

\begin{lemma}\label{lemma:translIsomnd}
  In the setting of Proposition \ref{prop:genExpSum}, let us assume that $\Fc_2$ is a Fourier sheaf, $f$ a polynomial of degree $d\ge 1$, $n=\Swan_\infty(\Fc_1)>d$, and $(n,d)=(d,p)=1$. If there is a geometric isomorphism of the form \eqref{eq:translatinIsom} with $a\neq 0$ for $\Gc=\FT_{\psi}(\Fc_2)$, then
\[\Swan_\infty(\Lc)\in\left\{0,1,\dots,\left\lfloor \frac{n}{n-d}\right\rfloor\right\}.\]
 If $n>2d$, there exists an additive character $\psi_1: \F_q\to\C^\times$ such that $\Lc\cong\Lc_{\psi_1}$.
\end{lemma}
\begin{proof}
  By \cite[7.7]{KatzESDE}, $\Fc_2$ has unique break $n/d$ at $\infty$, thus 
\[\Swan_\infty(\Fc_2)=(n/d)\rank(\Fc_2)=n.\] Moreover, $\Gc$ is lisse on $\A^1$. By Lemma \ref{lemma:translationInvariance} \ref{item:translationInvariance:SingL}, $\Sing(\Lc)\subset\{\infty\}$. We may assume that $\Lc$ is not geometrically trivial, the conclusions being clear otherwise. By Lemma \ref{lemma:translationInvariance} \ref{item:translationInvariance:EP}, it follows that $\Sing(\Lc)=\{\infty\}$ and $\Swan_\infty(\Lc)\ge 1$.

By \cite[7.4.1(1)]{KatzESDE}, $\Gc$ has unique break $\frac{n}{n-d}$ at $\infty$, with multiplicity
\[\frac{n-d}{n}\Swan_\infty(\Fc_2)=n-d.\]
The break $\frac{n}{n-d}$ is not an integer since we assume that $(n,d)=1$, and the first conclusion follows from Lemma \ref{lemma:translationInvariance} \ref{item:translationInvariance:UniqueBreak}. For the second one, note that $\frac{n}{n-d}<2$ if $n>2d$ and use Lemma \ref{lemma:FKMClass} \ref{item:FKMClass4}.
\end{proof}

The next lemma consequently considers isomorphisms of the form \eqref{eq:translatinIsom} when $\Lc$ is an Artin-Schreier sheaf.
\begin{lemma}\label{lemma:translIsomExplicit}
  In the setting of Proposition \ref{prop:genExpSum}, let us assume that $\Fc_2$ is a Fourier sheaf and that there is an isomorphism of the form \eqref{eq:translatinIsom} for $\Gc$ with $a\in\F_q^\times$ and $\Lc=\Lc_{\psi_1}$ for some additive character $\psi_1: \F_q\to\C^\times$. Then
  \begin{enumerate}
  \item $\Sing(\Fc_2)=\{\infty\}$ or $\A^1(\F_q)\subset\Sing(\Fc_2)$.
  \item If $f=X$, then either $\chi \neq 1$ and $g$ is constant, or $\chi=1$ and $h$ is a polynomial of degree at most $2$.
  \end{enumerate}
\end{lemma}

\begin{remark}
  Since we consider families of sheaves whose conductors are bounded uniformly from $q$, the condition $\A^1(\F_q)\subset\Sing(\Fc_2)$ is clearly exceptional.
\end{remark}

\begin{proof}
  Let $b\in\F_q$ such that $\psi_1(x)=\psi(bx)$ ($x\in\F_q$) and let us assume that we have a geometric isomorphism
  \[[+a]^*\Gc\cong\Gc\otimes\Lc_{\psi(bX)}\]
  with $a\in\F_q^\times$. Taking Fourier transform on both sides of the isomorphism and using that
  \begin{eqnarray*}
    {}[+a]^*\FT_\psi(\Fc)&\cong&\FT_\psi(\Fc\otimes\Lc_{\psi(aX)})\\
    \FT_\psi(\FT_\psi(\Fc)\otimes\Lc_{\psi(bX)})&\cong&[x\mapsto-b-x]^*\Fc
  \end{eqnarray*}
  for any Fourier sheaf $\Fc$, we get a geometric isomorphism
   \begin{equation}
    \label{eq:translIsomExplicit}
    \Fc_2\otimes\Lc_{\psi(aX)}\cong[+(-b)]^*\Fc_2.
  \end{equation}
  Then:
  \begin{itemize}
  \item If $b=0$, taking determinants shows that $a=0$.
  \item Since the Artin-Schreier sheaf is ramified at most at $\infty$, we have $\Sing(\Fc_2)\cap\A^1(\F_q)=(\Sing(\Fc_2)\cap\A^1(\F_q))+b$. If $b\neq 0$, this yields
  \[\Sing(\Fc_2)=\varnothing, \ \{\infty\},\text{ or }\A^1(\F_q)\subset\Sing(\Fc_2)\]
  because for any $y\in\F_q$, $b\in\F_q^\times$, the map $\F_q\to\F_q$, $x\mapsto y+xb$, is a bijection. By Lemma \ref{lemma:FKMClass}, $\Sing(\Fc_2)\neq\varnothing$ because we assume that $\Fc_2$ is geometrically irreducible and not geometrically trivial.
  \end{itemize}
  
  If $f=X$ and $b\neq 0$, the geometric isomorphism \eqref{eq:translIsomExplicit} becomes
  \[\Lc_{\psi(h(X)-h(X-b)+aX)}\cong\Lc_{\chi(g(X-b)/g(X))}.\]
  Since the Kummer sheaf is tame while the Artin-Schreier sheaf is not, we must have $\chi=1$ or $x\mapsto g(x-b)/g(x)$ constant on $\F_q$. If $\chi=1$, then
  \[x\mapsto h(x)-h(x-b)+ax\text{ is constant on }\F_q,\]
  i.e. $h(x)=-ab^{-1}x^2/2+ax/2+\text{constant}$. On the other hand, if $x\mapsto g(x-b)/g(x)$ is constant, then $g$ is constant.

  The case with a geometric isomorphism $[+a]^*\Gc\cong D(\Gc)\otimes\Lc_{\psi(bX)}$ is similar.
\end{proof}
\subsubsection{Sums of the form \eqref{eq:FouvryMichel}}

  \begin{proposition}\label{prop:FouvryMichelModelable}
  Let $f\in\Q(X)$ and let $Z_{f'}$ be the set of zeros of $f'$ in $\C$. We assume that either
  \begin{itemize}
  \item $(H)$: $k_f=|Z_{f'}|$ is even, $\beta=\sum_{z\in Z_{f'}} f(z)=0$, and if $s_1-s_2=s_3-s_4$ with $s_i\in f(Z_{f'})$, then $s_1=s_3,s_2=s_4$ or $s_1=s_2,s_3=s_4$.
  \item $(H')$: $f$ is odd, and if $s_1-s_2=s_3-s_4$ with $s_i\in f(Z_{f'})$, then $s_1=s_3,s_2=s_4$ or $s_1=s_2,s_3=s_4$ or $s_1=-s_4,s_2=-s_3$.
  \end{itemize}
  For $q$ large enough, there exists an $\ell$-adic sheaf $\Gc_{f,q}$ over $\F_q$ with trace function
  \[x\mapsto \frac{-1}{\sqrt{q}}\sum_{y\in\F_q}e\left(\frac{\tr(xf(y))}{p}\right)\hspace{0.2cm} (x\in\F_q).\]
  Moreover, there exist Weil numbers $\alpha_q\in\overline\Q_\ell$ of weight $0$ such that the family of $\ell$-adic sheaves $(\alpha_q\otimes\Gc_{f,q})_q$ is coherent, with monodromy group $\SL_{k_f}(\C)$ if $(H)$ holds, and $\Sp_{k_f}$ if $(H')$ holds (in which case $\alpha_q=1$).
\end{proposition}
\begin{proof}
  The construction of $\Gc_{f,q}$ is done in \cite[Theorem 7.9.4, Lemmas 7.10.2.1, 7.10.2.3]{KatzESDE} and the computation of monodromy groups in \cite[7.9.6, 7.9.7, 7.10]{KatzESDE}.

By Section \ref{subsec:arithmGeomMono}, we get $G_\geom=G_\arith=\Sp_{k_f}(\C)$ in the $(H')$ case. In the $(H)$ case, we use the determination (geometrically) of the determinant of $\Gc_{f,q}$ from \cite[7.10.4]{KatzESDE}: there is a geometric isomorphism
  \[\det(\Gc_{f,q})\cong \Lc_{\psi(-\beta X)}\otimes \Lc_{\chi},\]
  where $\chi=\chi_2^{k_f}$ for $\chi_2$ the character of order $2$ of $\closure{\F}_q^\times$ and $\beta$ is viewed in $\overline\F_q$. Under $(H)$ or $(H')$, this sheaf is geometrically trivial, and it suffices to apply \cite[Proposition 3.2.3]{FKMCours}.\\
  
  It remains to show the independence of shifts.  We consider the case of a geometric isomorphism
  \begin{equation}
    \label{eq:GcfTranslationInvariance}
    [+a]^*\Gc_{f,q}\cong\Gc_{f,q}\otimes\Lc
  \end{equation}
  for $\Lc$ a rank 1 sheaf and $a\in\F_q$, the argument with $D(\Gc_{f,q})$ being similar. We adapt the multiplicative case treated in the proof of \cite[Théorème 2.3]{MichelMinorationsSommesExp}. By Lemma \ref{lemma:translationInvariance} \ref{item:translationInvariance:SingL}, since $\Gc_{f,q}$ is lisse on $\G_m$, we must have $\Sing(\Lc)=\{0,-a,\infty\}$ or $\{0,-a\}$. Moreover, by \cite[7.5.4(5)]{KatzESDE}, the ramification of $\Lc$ at $0$ and $-a$ is tame. By \cite[7.9.4]{KatzESDE}, $\Gc_{f,q}$ as $I_\infty$-representation is
\[\Gc_{f,q}(\infty)\cong\bigoplus_{z\in Z_{f'}} \left(\Lc_{\psi(f(z)X)}\otimes \Lc_{\overline{\chi}_{z}(X)}\right)\]
where $\chi_{z}$ is a multiplicative character, and we view $Z_{f'}$ in $\overline\F_q$. Hence, by \cite[Lemma 1.3]{KatzGKM} all the breaks are at $1$ and as representations of the wild inertia group $P_\infty$, we have
\[\Gc_{f,q}(\infty)\cong\bigoplus_{z\in Z_{f'}} \Lc_{\psi(f(z)X)}.\]

We distinguish two cases:
  \begin{itemize}[itemsep=0.2cm]
  \item If $\infty\not\in\Sing(\Lc)$, then Lemma \ref{lemma:FKMClass} \ref{item:FKMClass3} implies that there is a multiplicative character $\chi_1$ such that
\[\Lc\cong\Lc_{\chi_1((X+a)/X)}.\]
Hence, there exists some $\beta\in\C$ of unit norm such that 
\[\beta\sum_{y\in\F_q} e\left(\frac{\tr((x+a)f(y))}{p}\right)=\sum_{y\in\F_q}e\left(\frac{\tr(xf(y))}{p}\right)\chi_1\left(\frac{x+a}{x}\right)\]
for all $x\in\F_q^\times$. If $a\neq 0$, taking $x=-a$ gives $\beta q=0$, a contradiction.
\item Assume that $\infty\in\Sing(\Lc)$. By \cite[Lemma 1.3]{KatzGKM}, $\Swan_\infty(\Lc)\in\{0,1\}$ because all the breaks of $\Gc_{f,q}$ at $\infty$ are at $1$. If $\Swan_\infty(\Lc)=1$, the break-depression lemma \cite[8.5.7]{KatzGKM} implies that $\Lc\cong(\text{tame at }\infty)\otimes\Lc_{\psi(bX)}$ for some $b\in\F_q^\times$. On the other hand, $\Lc$ is by definition tame at $\infty$ if $\Swan_\infty(\Lc)=0$. In both cases, the restriction of the isomorphism \eqref{eq:GcfTranslationInvariance} to $P_\infty$ gives
\[\bigoplus_{z\in Z_{f'}}\Lc_{\psi(f(z)(X+a))}\cong\bigoplus_{z\in Z_{f'}}\Lc_{\psi((f(z)+b)X)}\]
for some $b\in\F_q$. Thus the sets $\{f(z)(X+a) : z\in Z_{f'}\}$ and $\{(f(z)+b)X : z\in Z_{f'}\}$ are equal, which implies that $a=0$ (and $b=0)$.
  \end{itemize}
\end{proof}

\begin{remark}
  Lemma \ref{lemma:translIsomnd} does not apply here because $\Fc_1$ is trivial.
\end{remark}

\begin{examples}\label{ex:FouvryMichel}
  The following examples are given in \cite[p. 229]{MichelMinorationsSommesExp}, \cite[p. 7]{FouvryMichelSommes} and \cite[7.10]{KatzESDE}:
  \begin{enumerate}
  \item\label{item:FouvryMichel1} The polynomial $f=aX^{r+1}+bX$ with $a,b,r\in\Z$ and $ab\neq 0$ verifies $k_f=|r|$ and
    \[
    \begin{cases}
      (H)&\text{if }|r|\ge 3\text{ odd},\\
      (H')&\text{if }r\neq 0\text{ even.}
    \end{cases}
\]
  \item Let $g\in\Z[X]$ be monic of degree $r$ with full Galois group $\Sf_r$ (a ``generic'' condition by \cite{vdW34}), and let $f\in\Q[X]$ be the unique primitive of $g$ with $\sum_{i=1}^r f(\alpha_i)=0$, where $\alpha_1,\dots,\alpha_r$ are the zeros of $f$. Assuming that $r\ge 6$ is even, we have that $(H)$ holds for $f$ and $k_f=n$.
  \item For $n\ge 3$, the polynomial $f=X^n-naX$ satisfies $(H')$, and $k_f=n-1$.
  \end{enumerate}
\end{examples}

\subsubsection{Sums of the form \eqref{eq:expsumGen} with $f=X$, $\chi=1$, $h$ polynomial}

  \begin{proposition}\label{prop:expsumGenfXchi1hpol}
  Let $h=\sum_{i=0} a_nX^n\in\Z[X]$ be a polynomial of degree $n\ge 3$. For $p$ large enough (depending on $h$), there exists an $\ell$-adic sheaf $\Gc_{h,q}$ over $\F_q$ of rank $n-1$ corresponding to the trace function
  \[x\mapsto\frac{-1}{\sqrt{q}}\sum_{y\in\F_q} e\left(\frac{\tr(xy+h(y))}{p}\right) \ (x\in\F_q).\]
  If $a_{n-1}=0$ and $n\not\in\{7,9\}$, there exist Weil numbers $\alpha_q\in\overline\Q_\ell$ of weight $0$ such that the family $(\alpha_q\otimes\Gc_{h,q})_q$ is coherent, with monodromy group
  \begin{enumerate}
\item $\SL_{n-1}(\C)$ if $n-1$ is odd,
\item If $n$ is odd:
  \begin{itemize}
  \item $\Sp_{n-1}(\C)$ if $h$ has no monomial of even positive degree,
  \item $\SL_{n-1}(\C)$ otherwise.
  \end{itemize}
\end{enumerate}
 Moreover, $\alpha_q=1$ in the symplectic case.
\end{proposition}
\begin{proof}
  The construction of $\Gc_{h,q}$ and the computation of its geometric monodromy group can be found in \cite[7.12]{KatzESDE} and \cite{KatzMonodromyFamES}. In the symplectic case, Section \ref{subsec:arithmGeomMono} gives that $G_\geom=G_\arith=\Sp_{n-1}(\C)$. In the special linear case, the hypothesis $a_{n-1}=0$ implies that the geometric determinant of $\Gc$ is trivial by \cite[Section 7.12]{KatzESDE}, and the statement follows from Section \ref{subsec:arithmGeomMono}.
  
  The independence of shifts follows directly from Lemmas \ref{lemma:translIsomnd} and \ref{lemma:translIsomExplicit}, similarly to Kloosterman sheaves.
\end{proof}
\begin{example}
  For the Birch sums \eqref{eq:BS}, we have $h=X^3$ and the corresponding monodromy group is $\Sp_2(\C)=\SL_2(\C)$.
\end{example}

\subsubsection{Sums of the form \eqref{eq:expsumGen} with $f$ polynomial, $\chi\neq 1$}

  \begin{proposition}\label{prop:expSumfpolchineq1}
  Let
  \begin{itemize}
  \item $h\in\Q(X)$ with a pole of order $n\ge 1$ at $\infty$.
  \item $f\in\Z[X]$ nonzero of degree $d$ with $(d,n)=1$.
  \item $g\in\Q(X)$ nonzero.
  \item $\chi$ a character of $\F_q^\times$ of order $r\ge 2$, with the order of any zero or pole of $g$ not divisible by $r$.
  \end{itemize}
  For $p$ large enough (depending on $f,g,h$), there exists an $\ell$-adic sheaf $\Gc_q$ over $\F_q$ corresponding to the trace function \eqref{eq:expsumGen}. Assuming that $n>2d$, that $f,h$ are odd and that
 \begin{itemize}
 \item there exists $L\in\Q(X)$ even or odd with $L(x)^r=g(x)g(-x)$, 
 \item either $g$ is nonconstant or $h\not\in\Z[X]$,
 \item either $N=\rank(\Gc)\neq 8$ or $n-d\neq 6$,
 \end{itemize}
then there exist $\alpha_q\in\{\pm 1\}$ such that the family $(\alpha_q\otimes\Gc_q)_q$ is coherent, with monodromy group $\Sp_N(\C)$ if $L$ is odd (in which case $\alpha_q=1$) and $\SO_N(\C)$ if $L$ is even.
\end{proposition}
\begin{proof}
  The construction and the computation of the geometric monodromy group of $\Gc_q$ can be found in \cite[7.7, 7.13 ($\Sp$-example(2)) and 7.14 ($\O$-example(2))]{KatzESDE}. Section \ref{subsec:arithmGeomMono} show the existence of $\alpha_q\in\{\pm 1\}$ so that $G_\geom(\alpha_q\otimes\Gc_q)=G_\arith(\alpha_q\otimes\Gc_q)$ is as stated.

  We show the independence of shifts. Let us assume that there is a geometric isomorphism of the form \eqref{eq:translatinIsom} for $\Gc$ with $a\neq 0$. By Lemmas \ref{lemma:translIsomnd} and \ref{lemma:translIsomExplicit}, we have $\Sing(\Fc_2)=\{\infty\}$ or $\A^1(\F_q)\subset\Sing(\Fc_2)$. Since $\cond(\Fc_2)$ is bounded independently from $q$, the last possibility is excluded for $q$ large enough. Let us then assume that $\Sing(\Fc_2)=\{\infty\}$. Because $f$ is a polynomial, we have $\Sing(\Fc_1)\subset\{\infty\}$. Since the Kummer sheaf is tamely ramified everywhere while the Artin-Schreier sheaf is totally wild at all ramified points, this implies that $h\in\Z[X]$ and that $g$ is constant.
\end{proof}

\subsection{Families of hyperelliptic curves}
\begin{proposition}\label{prop:hyperEllipticFamily}
  Let $f\in\Z[X]$ be a squarefree polynomial of degree $2g\ge 2$. For $q$ large enough, we consider the family of smooth projective models of the affine hyperelliptic curves over $\F_q$ of genus $g$ given by
  \[X_z : y^2=f(x)(x-z),\]
  parametrized by $z\in\F_q$, which are nonsingular when $z\not\in Z_{f,q}$, for $Z_{f,q}\subset\overline\F_q$ the set of zeros of $f$ in $\F_q$. There exists a geometrically irreducible $\ell$-adic sheaf $\Fc_{f,q}$ over $\F_q$ of rank $2g$, with trace function
  \[t_\Fc(z)=\frac{q+1-|X_z(\F_{q})|}{q^{1/2}} \ (z\not\in Z_{f,q}).\]
  The family $(\Fc_{f,q})_q$ is coherent with monodromy group $\Sp_{2g}(\C)$.
\end{proposition}
\begin{proof}
  For the construction, see \cite[Section 10.1]{KatzSarnak91} or \cite[Section 4]{Hall08} (using middle-convolutions). Here, we moreover normalize with a Tate twist to get a sheaf of weight $0$. We have $\Sing(\Fc_{f,q})=\{\infty\}\cup Z_f$ and $\Fc_{f,q}$ is everywhere tame. In particular, $\cond(\Fc_{f,q})=2g+|Z_{f,q}|$.

  By \cite[Theorem 10.1.16]{KatzSarnak91}, the geometric monodromy group is symplectic. Since we normalized, \cite[Lemma 10.1.9]{KatzSarnak91} shows that the arithmetic monodromy group preserves the sames pairing (without normalization, it is a symplectic similitude with multiplicator $q$).

  It remains to show the independence of shifts. By \cite[10.1.13]{KatzSarnak91}, at any $z\in Z_{f,q}$ the quotient $V/V^{I_z}$ is the trivial (one-dimensional) $I_z$--representation, for $V=(\Fc_{f,q})_{\overline\eta}$. Let us assume that there exists an isomorphism of the form \eqref{eq:translatinIsom} for $\Fc_{f,q}$. By Lemma \ref{lemma:translationInvariance} \ref{item:translationInvariance:Invariants}, if $q$ is large enough, there exists $x\in\Sing(\Fc)\cap\A^1(\F_q)$ such that $V^{I_x}=0$, a contradiction.
\end{proof}

\bibliographystyle{alpha}
\bibliography{references}

\begin{thebibliography}{FKM15b}

\bibitem[BRR86]{NormAppr}
Rabi~N. Bhattacharya and Ramaswamy Ranga~Rao.
\newblock {\em Normal approximation and asymptotic expansions}.
\newblock Robert E. Krieger Publishing Co., 1986.
\newblock Reprint of the 1976 original.

\bibitem[DE52]{DavenportErdos}
Harold Davenport and Paul Erd\H{o}s.
\newblock The distribution of quadratic and higher residues.
\newblock {\em Publ. Math. Debrecen}, 2:252--265, 1952.

\bibitem[Del77]{DelEC}
Pierre Deligne.
\newblock {\em Cohomologie étale, séminaire de géométrie algébrique du
  {B}ois-{M}arie {SGA} 4$\frac{1}{2}$}, volume 569 of {\em Lecture notes in
  {M}athematics}.
\newblock Springer, 1977.

\bibitem[Del80]{Del2}
Pierre Deligne.
\newblock La conjecture de {W}eil. {II}.
\newblock {\em Publ. Math. Inst. Hautes Études Sci.}, 52(1):137--252, 1980.

\bibitem[DS94]{DiaconisShahshahani}
Persi Diaconis and Mehrdad Shahshahani.
\newblock On the eigenvalues of random matrices.
\newblock {\em J. Appl. Probab.}, 31:49--62, 1994.

\bibitem[FH91]{FulHar91}
William Fulton and Joe Harris.
\newblock {\em Representation theory}, volume 129 of {\em Graduate texts in
  {M}athematics}.
\newblock Springer, 1991.

\bibitem[FKM14a]{FKMCours}
\'Etienne Fouvry, Emmanuel Kowalski, and Philippe Michel.
\newblock Trace functions over finite fields and applications.
\newblock \url{https://people.math.ethz.ch/~kowalski/elements.pdf}, December
  2014.

\bibitem[FKM14b]{FKMPisa}
\'Etienne Fouvry, Emmanuel Kowalski, and Philippe Michel.
\newblock Trace functions over finite fields and their applications.
\newblock {\em Colloquium De Giorgi}, 2014.

\bibitem[FKM15a]{AlgebraicTwists}
\'Etienne Fouvry, Emmanuel Kowalski, and Philippe Michel.
\newblock Algebraic twists of modular forms and {H}ecke orbits.
\newblock {\em Geom. Funct. Anal.}, 25(2):580--657, 2015.

\bibitem[FKM15b]{FKMSumProducts}
\'Etienne Fouvry, Emmanuel Kowalski, and Philippe Michel.
\newblock A study in sums of products.
\newblock {\em Philos. Trans. A}, 373(2040), 2015.

\bibitem[FM02]{FouvryMichelRecherche}
\'Etienne Fouvry and Philippe Michel.
\newblock A la recherche de petites sommes d'exponentielles.
\newblock {\em Ann. Inst. Fourier (Grenoble)}, 52(1):47--80, 2002.

\bibitem[FM03]{FouvryMichelSommes}
\'Etienne Fouvry and Philippe Michel.
\newblock Sommes de modules de sommes d'exponentielles.
\newblock {\em Pacific J. Math.}, 209(2), 2003.

\bibitem[Gut05]{Gut}
Allan Gut.
\newblock {\em Probability: a graduate course}.
\newblock Springer texts in statistics. Springer, 2005.

\bibitem[Hal08]{Hall08}
Chris Hall.
\newblock Big symplectic or orthogonal monodromy modulo $\ell$.
\newblock {\em Duke Math. J.}, 141(1):179--203, 2008.

\bibitem[IK04]{IK04}
Henryk Iwaniec and Emmanuel Kowalski.
\newblock {\em Analytic number theory}.
\newblock Colloquium Publications. American Mathematical Society, 2004.

\bibitem[Kat87]{KatzMonodromyFamES}
Nicholas~M. Katz.
\newblock On the monodromy groups attached to certain families of exponential
  sums.
\newblock {\em Duke Math. J.}, 54(1), 1987.

\bibitem[Kat88]{KatzGKM}
Nicholas~M. Katz.
\newblock {\em Gauss sums, {K}loosterman sums, and monodromy Groups}, volume
  116 of {\em Annals of Mathematical Studies}.
\newblock Princeton University Press, 1988.

\bibitem[Kat90]{KatzESDE}
Nicholas~M. Katz.
\newblock {\em Exponential sums and differential equations}, volume 124 of {\em
  Annals of Mathematical Studies}.
\newblock Princeton University Press, 1990.

\bibitem[KS91]{KatzSarnak91}
Nicholas~M. Katz and Peter Sarnak.
\newblock {\em Random matrices, {F}robenius eigenvalues and monodromy},
  volume~45 of {\em Colloquium {P}ublications}.
\newblock American {M}athematical {S}ociety, 1991.

\bibitem[KS14]{KowSawin}
Emmanuel Kowalski and William~F. Sawin.
\newblock Kloosterman paths and the shape of exponential sums.
\newblock {\em Compos. Math.}, 2014.
\newblock To appear.

\bibitem[Lam13]{LamzShortSums}
Youness Lamzouri.
\newblock The distribution of short character sums.
\newblock {\em Math. Proc. Cambridge Philos. Soc.}, 155(2):207--218, 2013.

\bibitem[Lar90]{LarsenSatoTateNormal}
Michael Larsen.
\newblock The normal distribution as a limit of generalized {S}ato-{T}ate
  measures.
\newblock Unpublished note,
  \url{http://mlarsen.math.indiana.edu/~larsen/papers/gauss.pdf}, 1990.

\bibitem[LZ12]{LamzModm}
Youness Lamzouri and Alexandru Zaharescu.
\newblock Randomness of character sums modulo $m$.
\newblock {\em J. Number Theory}, 132(12):2779--2792, 2012.

\bibitem[Mac95]{Macdonald}
Ian Macdonald.
\newblock {\em Symmetric functions and {H}all polynomials}.
\newblock Oxford Mathematical Monographs. Oxford University Press, second
  edition, 1995.

\bibitem[Mic98]{MichelMinorationsSommesExp}
Philippe Michel.
\newblock Minorations de sommes d'exponentielles.
\newblock {\em Duke Math. J.}, 95(2), 1998.

\bibitem[MZ11]{MakZahShortSums}
Kit-Ho Mak and Alexandru Zaharescu.
\newblock The distribution of values of short hybrid exponential sums on curves
  over finite fields.
\newblock {\em Math. Res. Lett.}, 18(1):155--174, 2011.

\bibitem[PG16]{PG16}
Corentin Perret-Gentil.
\newblock {\em Probabilistic aspects of short sums of trace functions over
  finite fields}.
\newblock PhD thesis, \textsc{ETH Z\"urich}, 2016.

\bibitem[Pol14]{Polymath8a}
D.H.J. Polymath.
\newblock New equidistribution estimates of {Z}hang type.
\newblock {\em Algebra {N}umber {T}heory}, 8(9), 2014.

\bibitem[Pro90]{Pro90}
Robert~A. Proctor.
\newblock A {Schensted} algorithm which models tensor representations of the
  orthogonal group.
\newblock {\em Canad. J. Math.}, 42(1):28--49, 1990.

\bibitem[PV04]{PasturVasilchuk}
Luc Pastur and Vladimir Vasilchuk.
\newblock On the moments of traces of matrices of classical groups.
\newblock {\em Comm. Math. Phys.}, 252, 2004.

\bibitem[Ram95]{RamCharBrauer}
Arun Ram.
\newblock Characters of {B}rauer's centralizer algebras.
\newblock {\em Pacific J. Math.}, 169(1), 1995.

\bibitem[Reg81]{RegevYoungDiagrams}
Amitai Regev.
\newblock Asymptotic values for degrees associated with strips of {Y}oung
  diagrams.
\newblock {\em Adv. Math.}, 41(2):115--136, 1981.

\bibitem[Sag15]{Sage}
SageMath.
\newblock {\em The {S}age {M}athematics {S}oftware {S}ystem ({V}ersion 6.10)},
  2015.
\newblock \url{http://www.sagemath.org}.

\bibitem[Sel92]{Selberg92}
Atle Selberg.
\newblock Old and new conjectures and results about a class of {Dirichlet}
  series.
\newblock {\em Proceedings of the Amalfi Conference on Analytic Number Theory
  (Maiori, 1989) University of Salerno}, pages 367--385, 1992.

\bibitem[Ser89]{Serre89}
Jean-Pierre Serre.
\newblock {\em Abelian $\ell$-adic representations and elliptic curves},
  volume~7 of {\em Research Notes in Mathematics}.
\newblock Addison-{W}esley, 1989.

\bibitem[Sun86]{Sundaram}
Sheila Sundaram.
\newblock {\em On the combinatorics of representations of $\Sp(2n,\C)$}.
\newblock PhD thesis, Massachusetts {I}nstitute of {T}echnology, 1986.

\bibitem[Sun90]{Sundaram2}
Sheila Sundaram.
\newblock Orthogonal tableaux and an insertion algorithm for {$\SO(2n + 1)$}.
\newblock {\em J. Combin. Theory Ser. A}, 53(2):239--256, 1990.

\bibitem[vdW34]{vdW34}
Bartel~Leendert van~der Waerden.
\newblock Die {Seltenheit} der {Gleichungen} mit {Affekt}.
\newblock {\em Math. Ann.}, 109:13--16, 1934.

\end{thebibliography}

\end{document}